\documentclass{amsart}
\usepackage{amsfonts,amscd,amssymb,amsmath,amsthm}
\usepackage{graphicx,caption,subcaption,mathrsfs,appendix,centernot,mathtools}
\usepackage{xparse,tikz}
\usepackage{epstopdf}

\usepackage[english]{babel}
\usepackage[utf8]{inputenc}

\usepackage[colorlinks,citecolor=blue]{hyperref}

\usepackage[
    backend=biber,
    style=alphabetic,
    sortlocale=auto,
    natbib=true,
    url=false, 
    doi=true,
    eprint=true,
    safeinputenc=true,
    giveninits=true,
    isbn=false,
    maxbibnames=10
]{biblatex}
\usepackage{ctable,color}
\usetikzlibrary{arrows,chains,matrix,positioning,scopes,patterns}

\begin{document}

\newtheorem{theorem}{Theorem}[section]
\newtheorem{lemma}[theorem]{Lemma}
\newtheorem{lm}[theorem]{Lemma}
\newtheorem{corollary}[theorem]{Corollary}
\newtheorem{conjecture}[theorem]{Conjecture}
\newtheorem{cor}[theorem]{Corollary}
\newtheorem{proposition}[theorem]{Proposition}
\newtheorem{prop}[theorem]{Proposition}
\theoremstyle{definition}
\newtheorem{definition}[theorem]{Definition}
\newtheorem{example}[theorem]{Example}
\newtheorem{question}[theorem]{Question}
\newtheorem{claim}[theorem]{Claim}
\newtheorem{remark}[theorem]{Remark}


\newcommand{\R}{\mathbb{R}}
\newcommand{\C}{\mathbb{C}}
\newcommand{\Z}{\mathbb{Z}}
\newcommand{\Q}{\mathbb{Q}}
\newcommand{\E}{\mathbb E}
\newcommand{\N}{\mathbb N}
\newcommand{\X}{\mathbf X}

\renewcommand{\cosh}{\operatorname{ch}}
\renewcommand{\sinh}{\operatorname{sh}}

\newcommand{\del}{\ensuremath{\partial}}
\newcommand{\Def}{\ensuremath{:=}}

\newcommand{\TODO}[1]{\textbf{TODO:}{#1}}


\renewcommand{\Pr}{\mathbb{P}}
\newcommand{\as}{\text{a.s.}}
\newcommand{\Prob}{\Pr}
\newcommand{\Exp}{\mathbb{E}}
\newcommand{\expect}{\Exp}
\newcommand{\1}{\mathbf{1}}
\newcommand{\prob}{\Pr}
\newcommand{\pr}{\Pr}
\newcommand{\filt}{\mathscr{F}}
\DeclareDocumentCommand \one { o }
{%
\IfNoValueTF {#1}
{\mathbf{1}  }
{\mathbf{1}\left\{ {#1} \right\} }%
}
\newcommand{\Bernoulli}{\operatorname{Bernoulli}}
\newcommand{\Binomial}{\operatorname{Binom}}
\newcommand{\Beta}{\operatorname{Beta}}
\newcommand{\Binom}{\Binomial}
\newcommand{\Poisson}{\operatorname{Poisson}}
\newcommand{\Exponential}{\operatorname{Exp}}
\newcommand{\Unif}{\operatorname{Unif}}
\newcommand{\lawequals}{\overset{\mathcal{L}}{=} }


\newcommand{\Deg}{\operatorname{deg}}
\DeclareDocumentCommand \vso { o }
{%
\IfNoValueTF {#1}
{\mathcal{V}  }
{\mathcal{V}\left( {#1} \right) }%
}

\DeclareDocumentCommand \deg { O{ } }
{ \operatorname{deg}_{ #1 }}
\newcommand{\oneE}[2]{\mathbf{1}_{#1 \leftrightarrow #2}}
\newcommand{\ebetween}[2]{{#1} \leftrightarrow {#2}}
\newcommand{\noebetween}[2]{{#1} \centernot{\leftrightarrow} {#2}}
\newcommand{\Gap}{\ensuremath{\tilde \lambda_2 \vee |\tilde \lambda_n|}}
\newcommand{\dset}[2]{\ensuremath{ e({#1},{#2})}}


\newcommand{\Htwo}{ \mathbb{H}}
\newcommand{\Hd}{ \mathbb{H}^d}
\newcommand{\Etwo}{ \mathbb{E}}
\newcommand{\GWT}{ \mathcal{T}}
\newcommand{\HB}{ B_{\Htwo}}
\newcommand{\MB}{ B_{\mathbb{M}}}
\newcommand{\LB}{ B_{\mathscr{L}}}
\newcommand{\dH}{ d_{\Htwo}}
\newcommand{\dA}{ d_{\mathcal{A}}}
\newcommand{\PV}{ \mathcal{V}}
\newcommand{\VD}{ \mathscr{G}^\lambda }
\newcommand{\diamM}{ \operatorname{diam}_{\mathbb{M}} }
\newcommand{\diamH}{ \operatorname{diam}_{\Htwo} }
\newcommand{\dE}{d_{\Etwo}}
\newcommand{\Lattice}{ \Gamma }
\DeclareDocumentCommand \Vol { O{G} }{ \operatorname{Vol}_{#1}}
\newcommand{\VolH}{ \operatorname{Vol}_{\Htwo}}
\newcommand{\VolE}{\operatorname{Vol}_{\Etwo}}
\newcommand{\VolV}{\operatorname{Vol}_{\HPV}}
\newcommand{\VolD}{\operatorname{Vol}_{\HPD}}
\newcommand{\VolM}{ \operatorname{Vol}_{\Htwo}}
\newcommand{\convH}{ \operatorname{conv}_{\Htwo}}
\newcommand{\HPV}{ \mathscr{V}^\lambda}
\newcommand{\HPD}{ \mathscr{D}^\lambda}

\newcommand{\Del}{ \operatorname{Del}}
\newcommand{\Star}{ \operatorname{St}}
\newcommand{\GF}{\mathcal{G}_f}
\newcommand{\DeltaH}{\Delta_{\Htwo}}
\newcommand{\DeltaE}{\Delta_{\Etwo}}
\newcommand{\angleH}{\angle_{\Htwo}}
\newcommand{\angleE}{\angle_{\Etwo}}

\newcommand{\CDisc}{ \operatorname{CD}_{\Htwo}}
\newcommand{\CCtr}{ \operatorname{CC}_{\Htwo}}

\newcommand{\PPP}{ \Pi^\lambda }
\newcommand{\PP}{ \mathcal{S} }

\DeclareDocumentCommand \TFX { O{i} O{\pi_{\X}} O{f}  }
{ \Delta_{ {#2}, {#3} }({#1})}

\DeclareDocumentCommand \Filt { O{n} }
{ \mathscr{F}_{{#1}} }

\DeclareDocumentCommand \isol { O{i} O{S} }
{ \Delta_{#1} \left( {#2} \right) }

\newcommand{\aec}{i^*}
\DeclareDocumentCommand \islands {  O{i} }
{ \operatorname{IS}_{ {#1 }} }

\newcommand{\BHF}{ \mathcal{B}}

\title[Expansion in hyperbolic space]{Anchored expansion of Delaunay complexes in real hyperbolic space and stationary point processes}

\author{Itai Benjamini}
\address{Department of Mathematics, Weizmann Institute of Science}
\email{itai.benjamini@weizmann.ac.il}
\author{Yoav Krauz}
\address{Tel-Aviv University}
\email{yoavkrauz@mail.tau.ac.il}
\author{Elliot Paquette}
\address{Department of Mathematics, McGill University}
\email{elliot.paquette@gmail.com}
\date{\today}

\begin{abstract}
We give sufficient conditions for a discrete set of points in any dimensional real hyperbolic space to have positive anchored expansion.  The first condition is a anchored bounded density property, ensuring not too many points can accumulate in large regions.  The second is a anchored bounded vacancy condition, effectively ensuring there is not too much space left vacant by the points over large regions.  These properties give as an easy corollary that stationary Poisson--Delaunay graphs have positive anchored expansion, as well as Delaunay graphs built from stationary determinantal point processes.  

We introduce a family of stationary determinantal point processes on any dimension of real hyperbolic space, the Berezin point processes, and we partially characterize them.  We pose many questions related to this process and stationary determinantal point processes.
\end{abstract}
\maketitle

\section{Introduction}

While still in graduate school,
Harry Kesten initiated the study of random walk on groups and famously related nonamenability to strictly positive spectral gap
in a truly pioneering work \cite{Kesten}.
Decades later he was leading the study of random walk in random environment
and of the geometry of randomly diluted media, and in particular the theory of percolation.  
In the following note we study the geometry of random discrete triangulations of hyperbolic spaces and of the speed of random walk on them.
We prove a type of relaxed spectral gap, \emph{anchored nonamenability}, which is known to have consequences for the speed of random walk and for return probability estimates, a central topic in Kesten’s work. 

Much of what we do is inspired by the construction of probabilistic relaxations of \emph{lattices}.
Lattices in hyperbolic space (see \cite{Gelander} for an introduction, c.f.\ \cite{Gromov84,Cornulier}) are group theoretic constructions whose coarse geometric properties mirror the underlying space.  As a purely geometric consequence, any real hyperbolic space $\mathbb{H}^d$ contains collections of points $S$ whose \emph{Delaunay graphs}, defined below in Section \ref{sec:delaunay}, have a linear isoperimetric inequality:
\begin{equation}\label{eq:delaunay_expansion}
    \inf_{\substack{V \subset S \\ |V| <\infty}} \frac{|\partial V|}{|V|} > 0,
  \end{equation}
where $\partial V$ denotes the edges of this Delaunay graph with exactly one vertex in $V.$  This follows from the nonamenability of isometry group of $S$ which in turn is a relatively simple consequence of the nonamenability of the ambient space (\cite{Gelander}).

However lattices are also quite special, for example there are not arbitrarily fine lattices in $\mathbb{H}^d$ (in $2$ dimensions this can be seen from applying Gauss--Bonnet -- in higher dimensions this is the Kazhdan--Margulis theorem).  There is a larger class of quasisymmetric analogues, such as aperiodic tilings and \emph{quasi-lattices}, which we define presently, of the hyperbolic plane which are good coarse approximations of $\mathbb{H}^d$ (see \cite{MargulisMozes,Penrose,BlockWeinberger}).  

A subset $S \subset \mathbb{H}^d$ is called \emph{coarsely dense} if there is a real number $c >0$ so that every $y \in \mathbb{H}^d$ is less than distance $c$ from a point of $S.$  On the other hand, a subset $S \subset \mathbb{H}^d$ is called \emph{coarsely discrete} if for every $r > 0$ there is a $K_r$ so that
\[
  | S \cap B_r(y) | \leq K_r
\]
for every $y \in \mathbb{H}^d.$   
Define a \emph{quasi-lattice} $S \subset \mathbb{H}^d$ as a set 
which is both coarsely dense and coarsely discrete.  As a corollary of \cite[Theorem 3.1]{BlockWeinberger}, any quasi-lattice $S \subset \mathbb{H}^d$ is nonamenable in that \eqref{eq:delaunay_expansion} holds.

Another possible quasisymmetric generalization of a lattice is to replace $S$ by a random collection of points whose law is invariant under the isometries of $\mathbb{H}^d$ -- a \emph{stationary point process.}  As a principal motivating example, suppose that we consider a Poisson point process $\mathcal{X}$ with invariant intensity measure.  Will this too be nonamenable in the sense of \eqref{eq:delaunay_expansion}?

This is too much to ask.  Because many Poisson points could arrive in any small ball, and these points can be arranged in any way, there will be
somewhere in the Delaunay graph of $\mathcal{X}$ a finite subgraph isomorphic to any neighborhood in a $d$--dimensional Euclidean lattice.

However, there are relaxations of \eqref{eq:delaunay_expansion} that may still be satisfied by the Delaunay graph of $\mathcal{X}.$  One natural such condition is \emph{anchored nonamenability}.  Say that a connected graph $G$ has \emph{positive anchored expansion} if for some fixed vertex $\rho$ in $G,$
\begin{equation}\label{eq:anchored}
  \inf_{\substack{V \ni \rho \\ V \text{ connected }}} \frac{|\partial_{\text{out}} V|}{|V|} > 0,  
\end{equation}
and where $\partial_{\text{out}}V$ is the subset of $V$ with a neighbor outside of $V.$  Say that a graph has anchored nonamenability if it has positive anchored expansion. Anchored nonamenability yields some attractive corollaries for random walks and other random processes (see \cite{Thomassen, Virag, HSS} -- note that in the weak form that we have formulated it, some standard corollaries may only hold under additional control on the degrees of the vertices).  For unimodular random graphs embedded in $\mathbb{H}^d$, almost sure anchored nonamenability implies that random walk has positive hyperbolic speed almost surely.\footnote{For unimodular random graphs, anchored nonamenability implies the weaker invariant nonamenability.  When the degree of root of such an invariantly nonamenable graph has finite expectation, random walk will escape from the root with linear rate in any (stationary) pseudometric on the graph in which balls grow subexponentially.  In particular, random walk will have positive speed almost surely in the hyperbolic metric when run on a stationary point process (c.f.\ the proof of \cite[Corollary 1.10]{Paquette01} or \cite{AHNR1}).}

We will give a general criterion for any subset $S \subset \mathbb{H}^d$ to be \emph{anchored nonamenable}.  Moreover, this criterion will be satisfied for any stationary Poisson process on $\mathbb{H}^d.$

\subsection{Voronoi and Delaunay complexes} \label{sec:delaunay}

We recall here the definitions of the Voronoi tessellation and Delaunay complex for a discrete set of points $S \subset \Hd.$  The \emph{Voronoi tessellation} is a partition of $\Hd$ into cells, each of which is a hyperbolic polyhedron containing exactly one point of $S,$ its \emph{nucleus}.  For a point $p \in S,$ the Voronoi cell is given by
\[
\left\{ z \in \Htwo : \dH(z,p_0) = \min_{p \in \PP} \dH(z,p)\right\},
\]
where $\dH$ denotes distance in the hyperbolic metric.  

The \emph{Delaunay complex} is a dual cell complex to the Voronoi tessellation: a collection of points $U \subset S$ having $|U| = d+1$ is a $d$-dimensional simplex if and only if there exists an open ball $V$ disjoint from $S$ so that $\partial V \supset U.$  When the points $S$ are in generic position, meaning there is no sphere containing more than $d+1$ points of $S,$ then the Delaunay complex canonically embeds into $\Hd$ by representing each $d$-dimensional simplex by a hyperbolic simplex with vertices from $S.$ 
The Delaunay graph (the $1$--skeleton of the Delaunay complex) can be succinctly represented by the edge set on $S$ in which two vertices $x$ and $y$ are adjacent if and only if there is an open ball $V$ disjoint from $S$ with closure containing $x$ and $y.$

\begin{figure}
  \begin{subfigure}[t]{0.30\textwidth}
    \includegraphics[width=\textwidth]{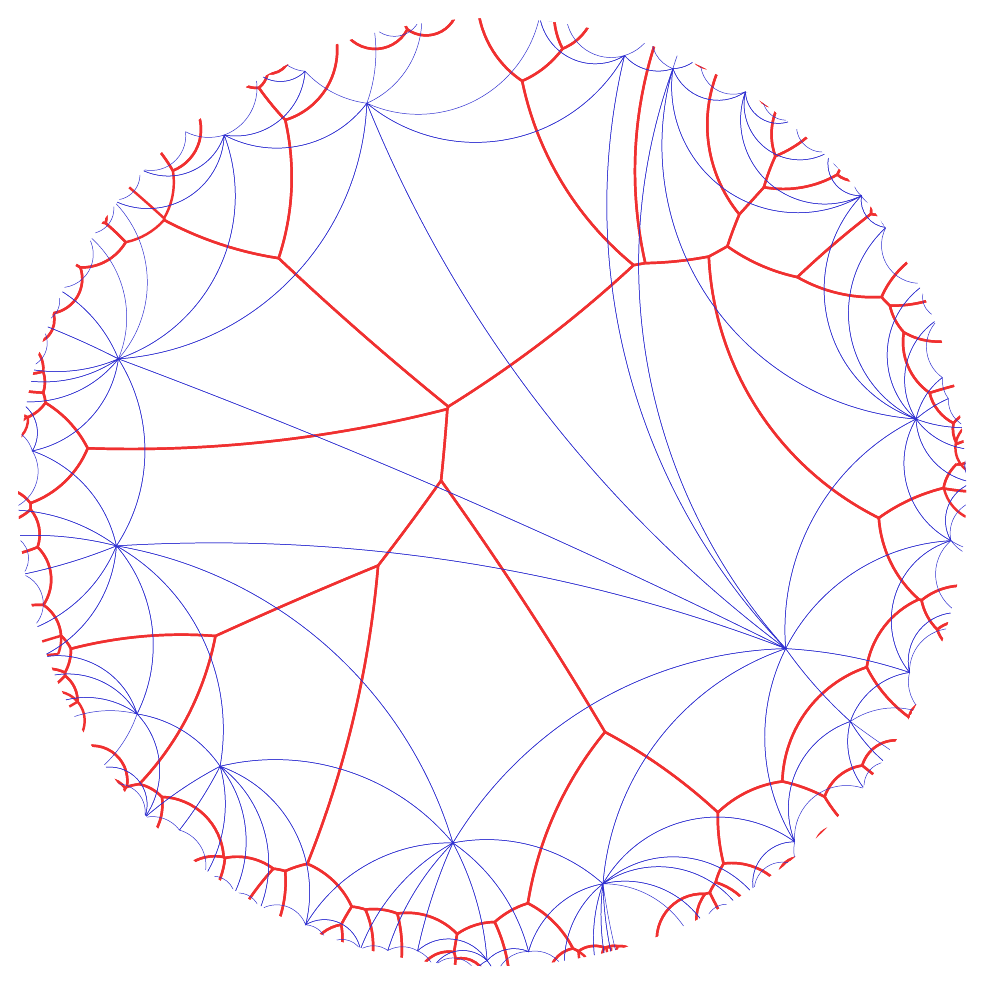}
    \caption{Berezin $s=2.44.$}
  \end{subfigure}
  \begin{subfigure}[t]{0.30\textwidth}
    \includegraphics[width=\textwidth]{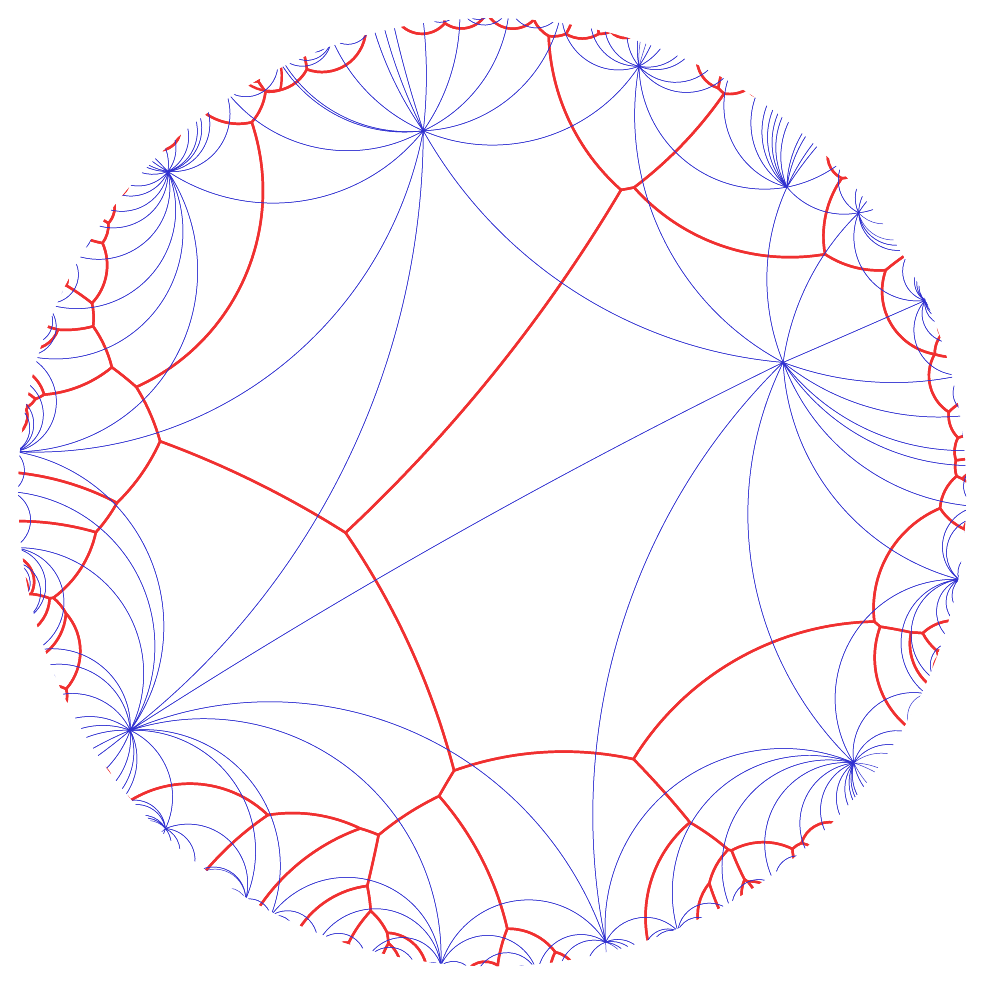}
    \caption{Hyperbolic GAF zeros.} 
  \end{subfigure}
  \begin{subfigure}[t]{0.30\textwidth}
    \includegraphics[width=\textwidth]{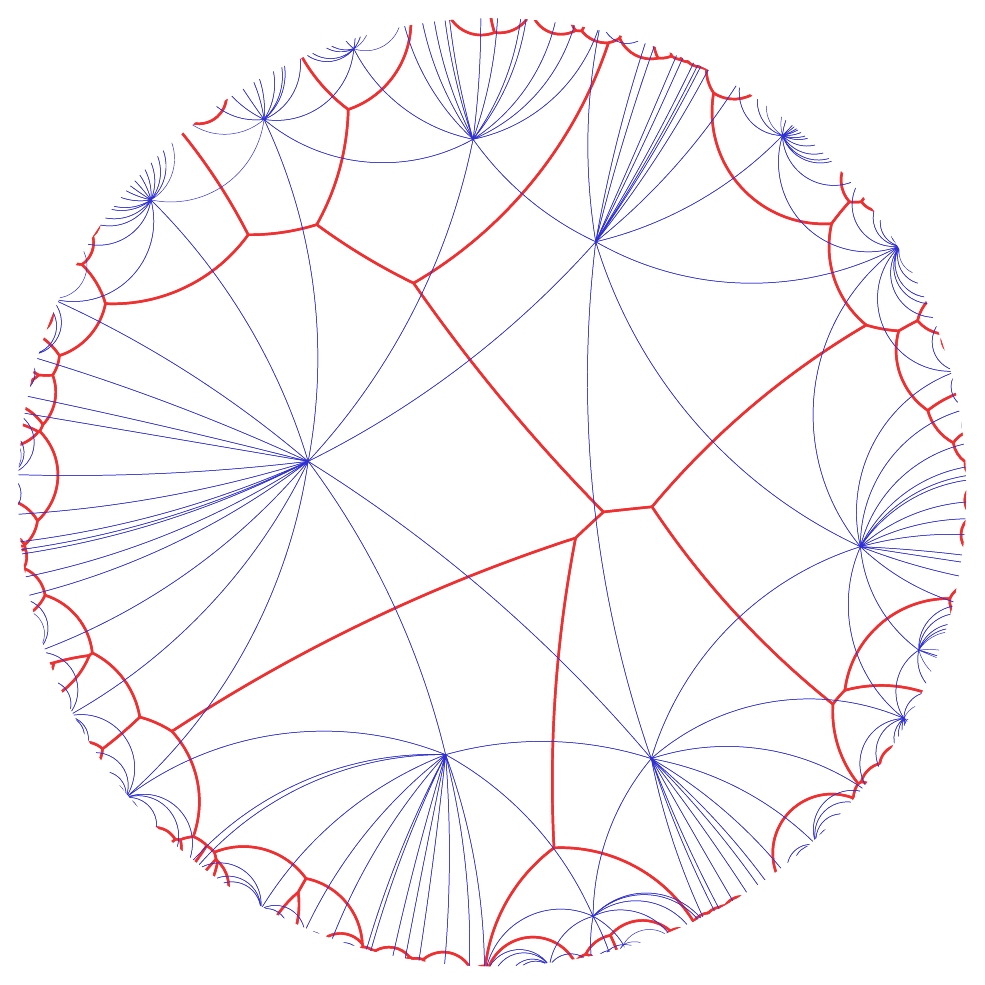}
    \caption{Poisson.} 
  \end{subfigure}
  \caption{Three simulations of Voronoi (red) and Delaunay (blue) in a ball (Euclidean radius $0.99$) in the Poincar\'e model.  Nuclei are the vertices of the blue complex.  The intensities of Berezin (see the definition in Section 5) and Poisson are chosen to match the GAF zeros. See also Figure \ref{fig:berezin}.
}
  \label{fig:comparison}
\end{figure}

\subsection{Main result} \label{sec:main}

To formulate our result, we will need to make use of a fixed cocompact lattice $\Lattice$ (see for example~\cite[Theorem 18.45]{WitteMorris}).
These guarantee the existence of a Voronoi tessellation of $\Hd$ into compact, convex polyhedra so that the isometries of $\Hd$ act transitively on the set of polyhedra.  As we will typically need to refer to the cells of the tessellation, we will use $\Lattice$ to denote the collection of these polyhedra.  

We let $A \in \Lattice$ be any fixed cell.  Say that two cells in $\Lattice$ are adjacent if their intersection has codimension-$1.$  Let $\mathcal{L}_A$ be the set of connected finite subsets of $\Lattice$ containing $A \in \Lattice.$
We say that a set of points $\PP \subset \mathbb{H}^d$ has \emph{anchored bounded density}
if
\begin{equation} \label{eq:bmd}
    \sup_{ J \in \mathcal{L}_A} \frac{|\PP \cap (\cup_{B \in J} B)|}{|J|} < \infty,
\end{equation}
Conversely, we say that a set of points $\PP \subset \mathbb{H}^d$ has \emph{anchored bounded vacancy} 
if there is an $M>0$ so that
\begin{equation} \label{eq:bmc}
    \inf_{ J \in \mathcal{L}_A} \frac{|\{B \in \Lattice: \dH(B,J) \leq M, |B \cap \PP| > 0\}|}{|J|} > 0.
\end{equation}

Note both of these definitions are independent of the choice of the base cell $A$ as well as the choice of cocompact lattice $\Lattice.$  The lattice $\Lattice$ could just well be replaced by the Voronoi tessellation of any quasi--lattice as well.  We further observe that any collection of points $\PP \subset \mathbb{H}^d$ that has both anchored bounded density and vacancy still have both of these properties after \emph{thinning}, i.e.\ randomly and independently removing each point of $\PP$ with any fixed probability $p > 0.$

Our main theorem is:
\begin{theorem}\label{thm:bmo_lattice}
  Suppose that $\PP \subset \mathbb{H}^d$ has both anchored bounded density and anchored bounded vacancy, then the Delaunay graph on $\PP$ has anchored nonamenability.
\end{theorem}

A Poisson point process satisfies both of these conditions easily (see Lemma \ref{lem:ppp}), from which we derive:
\begin{corollary}\label{cor:ppp_lattice}
  Let $\PPP$ be a Poisson point process on $\mathbb{H}^d$ with intensity measure which is a positive multiple $\lambda$ of hyperbolic volume measure. The Poisson--Delaunay graph with nuclei $\PPP$ is almost surely anchored nonamenable.
\end{corollary}
\noindent It is also easy to verify that any \emph{determinantal point process} with invariant intensity measure has anchored bounded density and anchored bounded vacancy (using negative association and standard tail properties of these properties, see Proposition \ref{prop:generalities}, c.f.\ \cite[Theorem 6.5]{Lyons}).  We will give a construction of a family of stationary determinantal points processes in all dimensions of real hyperbolic spaces in Section \ref{sec:hbr}.

\subsection{Controlling the degrees}

The isoperimetric constant we use in \eqref{eq:anchored} is not the one which is typically useful for finding corollaries of random walk (c.f.\ \cite{Virag,BenjaminiPaquettePfeffer01}).  If we define for a finite set of vertices $V$ in graph $G$, $\Vol(V)$ as the sum of the degrees of the vertices in $V,$ then we can say that a connected graph $G$ has \emph{positive strong anchored expansion} if
\begin{equation}\label{eq:sanchored}
  \inf_{\substack{V \ni \rho \\ V \text{ connected }}} \frac{|\partial_{\text{out}} V|}{\Vol(V)} > 0,  
\end{equation}
and say that $G$ has \emph{strong anchored nonamenability.}
To derive this from \eqref{eq:anchored}, it suffices to know that $\Vol(V)$ is comparable to $|V|.$  In fact it suffices to show that the number of internal edges $\Vol[V](V)/2$ is comparable to $|V|$ in the sense that
\begin{equation}\label{eq:degree}
  \sup_{\substack{V \ni \rho \\ V \text{ connected }}} \frac{\Vol[V](V)}{|V|} < \infty.  
\end{equation}

We give sufficient conditions for \eqref{eq:degree} to hold.  
We again let $\Lattice$ be any cocompact lattice and $A \in \Lattice$ be any fixed cell.  
We say that for some $p \geq 1$ a set of points $\PP \subset \mathbb{H}^d$ has \emph{anchored $p$-exponential density}
if
\begin{equation} \label{eq:bpmd}
    \sup_{ J \in \mathcal{L}_A} \frac{|\PP \cap (\cup_{B \in J} B)|^p}{|J|} < \infty.
\end{equation}
Note that for any $p > 1,$ this is strictly stronger than anchored bounded density \eqref{eq:bmd}.

Likewise, we say that for any $p>0,$ $\PP \subset \mathbb{H}^d$ has \emph{anchored $p$-exponential vacancy} if
\begin{equation} \label{eq:bpmv}
    \sup_{ J \in \mathcal{L}_A} \frac{
      \sum_{X \in J}e^{p(d-1)\dH(\PP,X)}
  }{|J|} < \infty.
\end{equation}
This condition implies anchored bounded vacancy for any $p > 0.$  We note that in $\mathbb{H}^d,$ the volume of a ball of radius $r$ has exponential growth rate $e^{r(d-1)},$ for which reason we have included the constant $(d-1).$

These conditions together imply that the degree sequence is in control.
\begin{theorem}\label{thm:bmo_lattice2}
  Suppose that for some $p,q > 0$ with $\frac{1}{p} + \frac{1}{q} < \tfrac12,$ $\PP \subset \mathbb{H}^d$ has both anchored $p$-exponential density and anchored $q$-exponential vacancy, then the Delaunay graph on $\PP$ has strong anchored nonamenability.
\end{theorem}
\noindent Indeed, we in fact show that under these two conditions, \eqref{eq:degree} holds. 

It is simple to check that some translated lattices satisfy the conditions of Theorem \ref{thm:bmo_lattice2}.  Consider taking a fixed lattice $\Gamma$ which we identify with a collection of points in $\mathbb{H}^d.$  
Let $\left\{ \phi_x : x\in \Gamma \right\}$ be the canonical isometric involutions of $\mathbb{H}^d$ interchanging $x$ and $0.$ 
Then define an iid family of automorphisms of $\mathbb{H}^d$ $\left\{ \psi_x : x \in \Gamma \right\}$ and define the \emph{translated lattice} $\hat{\Gamma} = \left\{ \phi_x(\psi_x(\phi_x(x))) : x \in \Gamma \right\}.$  Provided that $\dH(0,\psi_x(0))$ has sufficiently fast tail decay, then $\hat{\Gamma}$ will satisfy both anchored $p$-exponential density and bound $p$-exponential vacancy for any $p > 1.$  We do not go further into details.

The Poisson process does not satisfy this pair of conditions for any admissible pair $p,q.$  The Poisson process will have anchored $1$-exponential density but not anchored $p$-exponential density for any $p >1.$  Similarly, it has anchored $p$-exponential vacancy for $p < 1$ and not for $p > 1.$  We do not know if there is any stationary determinantal process which satisfy these conditions. 
\begin{question}
  Is there a stationary determinantal point process satisfying anchored $p$-exponential vacancy for any $p > 1$? Is there one satisfying anchored $p$-exponential density for any $p > 1$?
\end{question}
\noindent See Section \ref{sec:hbr} for a greatly expanded discussion of stationary determinantal point processes on real hyperbolic space and related questions.

%

\subsection{Related work and discussion}

\subsubsection*{Poisson Voronoi tessellations}

We give a criterion that shows that the Poisson--Delaunay graph almost surely is anchored nonamenable in any dimension of hyperbolic space. In \cite{BenjaminiPaquettePfeffer01}, it is shown that in dimension $2,$ the Poisson--Delaunay graph has a slightly stronger version of anchored nonamenability almost surely, but the arguments were highly specific both to the plane and $\PP=\PPP.$  Corollary \ref{cor:ppp_lattice} answers Conjecture 1.8 in that paper.  In \cite{Paquette01}, it is shown that the Delaunay graph of any stationary point process $\mathcal{S} \subset \mathbb{H}^d$ has a type of nonamenability in law, and that argument extends to other nonpositively curved spaces; it is also shown there that simple random walk on the Delaunay graph has positive hyperbolic speed.  For the case of the \emph{Poisson} Delaunay graph, it is shown there that random walk has positive \emph{graph} speed.  

Poisson--Voronoi tessellations have been the setting for interesting percolation theory and other invariant probability, starting with \cite{BenjaminiSchramm01}.  Very recently, it was established that in the $\lambda\to\infty$ limit, the critical value for Bernoulli site percolation tends to $\tfrac 12$ \cite{HansenMuller}, consistent with percolation on the Euclidean Poisson--Voronoi tessellation.  Indeed one attractive feature of the Poisson--Voronoi tessellation on $\PPP$ is that the parameter $\lambda$ is a continuously tunable (inverse) curvature parameter.

\subsubsection*{Other questions}

We have given a sufficient condition for a collection of points $\PP$ to have an anchored nonamenable Delaunay graph.  Having too many points of $\PP$ in a small region would seem to necessarily contradict that $\PP$ is anchored nonamenable, as Euclidean geometry would then take over.  On the other hand, having points at great distance would only seem to increase the boundary size of $\PP,$ suggesting that the anchored bounded vacancy is possibly unnecessary.
\begin{question}
  Suppose $\PP$ is a collection of points whose Voronoi complex has all bounded cells and which is anchored nonamenable.  Does it follow that $\PP$ has anchored bounded density?  Conversely, is there a $\PP \subset \mathbb{H}^d$ whose Voronoi complex has all bounded cells but which does not have anchored bounded vacancy?
\end{question}

The Voronoi tessellation is only one way to construct a tiling from a collection of points.  Particularly in $2$--dimensions, circle packing is a very powerful tool for describing embeddings of abstract tilings \cite{AHNR1,AHNR2,BT}.
\begin{question} 
  Suppose that an anchored nonamenable, bounded degree graph $G$ is the dual of a circle packing in $\mathbb{H}^d.$  Does it follow that the Delaunay graph on the vertices of $G$ is anchored nonamenable?
\end{question}

While there are curvature constraints on the types on the fineness of lattices in hyperbolic space, perhaps there are less symmetric sets without this constraint which still look symmetric in a coarse sense.  
\begin{question}
  Is there for any $c>0$ a collection of points $\PP \in \mathbb{H}^d$ so that $\PP$ is $c$-coarsely dense, and so that for each pair $\left\{ x,y \right\},$ there is a bilipschitz map $\phi$ preserving $\PP$ and interchanging $x$ and $y?$  If, conversely, we ask that the bilipschitz constant is sufficiently close to $1$, is $c$ bounded from below? 
\end{question}

\subsubsection*{Organization}

In Section \ref{sec:hvstructure} we give some basic structure of Voronoi complexes in hyperbolic space.  In Section \ref{sec:discretizations}, we prove Theorem \ref{thm:bmo_lattice}.  In Section \ref{sec:ppp}, we show Corollary \ref{cor:ppp_lattice}.  Finally in Section \ref{sec:hbr}, we introduce the Berezin point process and discuss stationary determinantal point processes in greater depth. 

\section{Convexity and the structure of hyperbolic Voronoi complexes}
\label{sec:hvstructure}

A subset $V \in \Htwo^d$ is convex if for any two points in $V,$ the hyperbolic geodesic connecting them is contained wholly in $V.$
For a set $S \subset \Htwo^d,$ the convex hull $\convH(S)$ is the intersection of all convex sets containing $V.$  
Much as in Euclidean space, it is also the intersection of hyperbolic half--spaces containing $S$.  Indeed, in the Klein--model of $\Htwo^d,$ in which a metric $\rho$ is placed on the unit ball of $\R^d$ so that $\rho$--geodesics are Euclidean geodesics, hyperbolically convex sets coincide with Euclidean convex sets.  Hence, for example, a convex hull of a finite set of points is still the  intersection of a family of finitely many half--spaces containing that set.  

Hence for any discrete set of points $S \subset \Htwo^d,$ the Voronoi complex with nuclei given by $S$ is, as in Euclidean space, a partition of the space into convex cells which overlap on lower dimensional faces. Additionally, as with Euclidean space, a Voronoi cell is unbounded if and only if its nucleus is an extreme point of the convex hull of $S.$  However, in hyperbolic space, horoballs (limits of balls with radius tending to infinity that are tangent to a point) and half--spaces do not coincide, and hence we are led to some additional characterizations of hyperbolic convex hulls, as elaborated in the next two lemmas.

\begin{lemma}
\label{lem:unbounded}
Let $S\subset\Htwo^d$ a discrete set of points, and let $p\in S$. Then, in the Voronoi complex corresponding to $S$, the cell with nucleus $p$ is unbounded if and only if there exists a horoball $V$ so that $V \cap S = \emptyset$ and so that $p \in \partial V.$
\end{lemma}
\begin{proof}
	The cell with nucleus $p$ is unbounded if and only if there is a sequence of balls $\left\{ \HB(q_n,r_n) \right\}_{n=1}^\infty \subset \Hd$ with $r_n \to \infty$ so that $p \in \partial \HB(q_n,r_n)$ but $\HB(q_n,r_n) \cap S = \emptyset.$  By compactness, we may extract a subsequence $n_k$ of balls so that the normal vectors at $p$ converge to some $b$.  It then follows that the union $\cup_{k=1}^\infty \HB(q_{n_k},r_{n_k})$ is disjoint from $S.$  As this union of balls contains the horoball $V$ with $p \in \partial V$ and with normal vector $b$ at $p,$ the proof is complete.
\end{proof}

As a corollary, we see that:
\begin{proposition}\label{prop:finitecells}
For a discrete point set $\PP$ which has anchored bounded vacancy, every Voronoi cell in the complex with nuclei $\PP$ is finite.
\end{proposition}
\begin{proof}
  Suppose that $\PP$ is a discrete point set whose Voronoi complex has an unbounded cell with nucleus $p.$  Then there is a horoball $V$ tangent to the ideal boundary at some point $\omega.$  Let $\gamma$ be a geodesic from $p$ to $\omega.$  Then any fixed size tubular neighborhood $N$ of $\gamma$ is eventually contained in $V$ in that $V^c \cap N$ is compact.  

  We may take $A \in \Lattice$ in the definition of anchored bounded vacancy to be the cell containing $p.$   Let $J_k$ be an increasing sequence in $\mathcal{L}_A$ with $|J_k| = k$ and every cell of $J_k$ intersecting $\gamma.$  Then for any $M > 0,$ the collection of $\Lattice$ cells near to $J_k$ by $\dA$ distance $M$ is contained in some tubular neighborhood $N$ of $\gamma.$  Hence all but finitely many cells are contained in $V,$ and consequently
  \[
    \lim_{k \to \infty} 
    \frac{|\{B \in \Lattice: \dH(B,J_k) \leq M, |B \cap \PP| > 0\}|}{|J_k|} = 0.
  \]
\end{proof}

We can also use Lemma \ref{lem:unbounded} to give two alternative representations of a convex hull of a set of points in hyperbolic space.
We define the \emph{Voronoi boundary} of a discrete set of points $S$ to be those points in the Voronoi complex with nuclei $S$ whose cells are unbounded.
\begin{lemma}
	\label{lem:QST}
For each point \(w\) on the ideal boundary of $\Hd$, let \(H_w\) be the maximal horoball that passes through \(w\) that is disjoint from \(S\).
Let $Q$ be the compact set $\cap_w H_w^c,$ with the intersection over the entire boundary of $\Hd$.
Let $T\subset S$ be the Voronoi boundary of $S$. Then
\[
\convH(Q)=\convH(S)=\convH(T).
\]
\end{lemma}
\begin{proof}
The inclusions $\convH(Q)\supset\convH(S)\supset\convH(T)$ follow trivially from the inclusions $Q\supset S \supset T$. Hence to prove the lemma, it is enough to show that $\convH(Q)\subset\convH(T).$ 
Let $p$ be an extreme point of $Q,$ so that there exists a hyperplane $L$ such that $L\cap Q \supset\left\{p\right\}$ and so that one of the closed half--spaces bounded by $L$ contains $Q$. Suppose on way to a contradiction that $p\not\in S.$
Let $V$ be the horoball that is tangent to $L$ at $p$ and that is disjoint from $Q.$  As $p$ is separated from $S,$ it is possible to slightly enlarge $V$ to $V'$, another horoball tangent to the ideal boundary at the same point that still is disjoint from $S.$  As $V' \supset \overline V,$ we have that $p \in Q^c,$ a contradiction.  Hence $p \in S.$  By Lemma~\ref{lem:unbounded}, we have that $Q \cap S \subset T,$ so that we have shown that all the extreme points of $Q$ are in $T.$  Hence $\convH(Q) \subset \convH(T).$
\end{proof}

A major difference between Euclidean Voronoi complexes and hyperbolic Voronoi complexes is that in a hyperbolic Voronoi complex with nuclei $S,$ points that are close to the boundary of $\convH(S)$ must actually be close to the Voronoi boundary of $S.$

\begin{proposition}
	\label{prop:convexhull}
	For any $D> 0$ there is an $R>0$ so that 
	for all $S \subset \Hd$ finite,
	\[
		\bigcup_{p\in S} \HB\left(p,D\right)
		\subset \convH(T) \cup \bigcup_{p\in T} \HB\left(p,R\right),
	\]
	where $T \subset S$ is the Voronoi boundary of $S.$  
\end{proposition}
\begin{proof}
  Let $x \in S$ be arbitrary.  We will take $R > D,$ and so if $x \in T$ there is nothing to show.  
  Likewise if $\HB(x,D) \subset \convH(T),$ there is nothing to show, and so there is a supporting hyperbolic half--space $L \supset \convH(S)$ not containing $\HB(x,D).$  Let $w$ be the point on the ideal boundary of $\Hd$ which is the endpoint of the geodesic ray from $x$ that is normal to $L,$ and let $p$ the point of intersection between this ray and $L.$  Let $H_w(y)$ be the open horoball defined by the ray from $y$ to $w.$   Note that $H_w(p) \subset L^c$ and therefore contains no points of $S.$  However, $H_w(x)$ is a horoball whose closure contains $x$ and by the assumption that $x\not\in T,$ it must be that $H_w(x)$ intersects $S.$  Hence, there is a point $y$ on the geodesic $[p,x]$ that is the closet point to $p$ at which the closure of $H_w(y)$ intersects $S.$  Let $t \in S$ be such a point, and note that $H_w(y)$ is an open horoball whose closure contains $t$ and is disjoint from $S$, and hence $t \in T.$  This implies that $T$ intersects $H_w(x) \cap L.$  The diameter of $H_w(x) \cap L$ can be controlled solely as a function of $D,$ which completes the proof.
\end{proof}

Finally, we show that to estimate the number of unbounded cells in a hyperbolic Voronoi complex, then it suffices to estimate, up to constants, the volume of a neighborhood of those points.

\begin{lm}
\label{lm:greatlemma}
There is a constant
For every $D > 0$ there is an $\beta$ so that for any finite set $S\subset\Htwo^d$ the Voronoi boundary $T$ of $S$ satisfies
\[
  {\VolH( \cup_{p \in S} \HB(p,D) )} \leq \beta |T|.
\]
\end{lm}
This lemma is effectively an immediate corollary of Proposition~\ref{prop:convexhull} and the following proposition, proved in \cite{BE}:
\begin{prop}
  For any finite set $S \subset \Htwo^d,$
  \[
    \VolH(\convH(S)) \leq \alpha_d |S|,
  \]
  where $\convH(S)$ denotes the hyperbolic convex hull $S$, and $\alpha_d$ is a constant depending only on $d$.
  \label{prop:BE}
  \end{prop}
\begin{proof}[Proof of Lemma~\ref{lm:greatlemma}]
  Using Proposition~\ref{prop:convexhull} and Proposition~\ref{prop:BE}, letting $T$ be the Voronoi boundary of $S$ and letting $o \in \Hd$ be any point,
  \[
    \begin{aligned}
   \VolH( \cup_{p \in S} \HB(p,D) ) 
    & \leq \VolH( \convH(T)) + |T|\VolH(\HB(o,R)) \\
    &\leq (\alpha_d + \VolH(\HB(o,R)))|T|.
    \end{aligned}
  \]
  Thus the bound follows with  $\beta =(\alpha_d + \VolH(\HB(o,R))).$
\end{proof}

\section{Lattice discretizations of Voronoi complexes}
\label{sec:discretizations}

We suppose that $\PP$ is a countable subset of $\Hd$ so that the Voronoi complex on $\mathcal{S}$ has only finite cells.  We will now formulate conditions under which we can show the Delaunay graph of $\PP$ has anchored nonamenability.

We again let $\Lattice$ be any cocompact lattice, which we once more identify with the cells in a Voronoi tessellation with nuclei given by the orbit of $\Lattice$ of a point.  By transitivity, we have that the volumes of the cells are equal, and we let $V$ denote that volume.  We also let $D$ be the diameter of one (hence all) of these cells.

We would like to say that connected sets of points in $\PP$ resemble lattice animals in $\Lattice$ in some way.
We let $\mathcal{A}$ be the graph with vertex set $\Lattice$ and edges between any $A,B$ having $\dH(A,B) \leq 2D.$
For any set $S$ of points in $\Htwo^d$, let
\begin{align*}
\mathcal{I}(S)&:=
\left\{A \in \Lattice:\left|A \cap S\right|>0\right\}, \\
\mathcal{E}(S)&:=
\left\{A \in \Lattice:\left|A \cap S\right| = 0, \text{$A$ intersects an $\mathcal{S}$--Voronoi cell with nucleus in $S$}\right\}. 
\end{align*}

\begin{lemma}
\label{lem:vboundary}
  Suppose that $\PP$ is a discrete set of points in $\Htwo^d$ for which all Voronoi cells are bounded.
There is a $C>0$ so that
for any finite set $S\subset \PP$, we have
\[
|\del_{\text{out}} S|\geq C |\mathcal{I}(S)|
\]
\end{lemma}
\begin{proof}[Proof of Lemma~\ref{lem:vboundary}]
  By assumption, there are no unbounded cells in the Delaunay complex on $\PP.$
Therefore, any cell in the Voronoi boundary of $S$ is in $\del_{\text{out}} S$. 

Note that $\mathcal{I}(S)$ is contained in the $D$--neighborhood of $S,$ and $\VolH( \cup_{A \in \mathcal{I}(S)} A) = V|\mathcal{I}(S)|.$
Hence by Lemma~\ref{lm:greatlemma},
\[
  |\del_{\text{out}} S|\geq |\mathcal{I}(S)|V/R.
\]
\end{proof}


We also show that Voronoi complexes can be discretized to form connected subsets of $\mathcal{A}$ in the sense of the following lemma.

\begin{lemma}
\label{lem:connected}
For any $S \subset \PP$ that is connected in the Delaunay graph on $\PP$, 
then $\mathcal{E}(S) \cup \mathcal{I}(S)$
is connected in $\mathcal{A}.$
\end{lemma}
\begin{proof}[Proof of Lemma~\ref{lem:connected}]
  Let $p \in S$ be arbitrary, and let $X$ be the $\PP$--Voronoi cell $X$ with nucleus given by $p.$  
  Let $h$ be any point in $X,$ and let $L_h$ be any cell of $\Gamma$ in $X.$  It suffices to show that $L_h$ is connected to $L_p \in \mathcal{I}(\{p\})$ in $\mathcal{A}.$

  Let $\gamma$ be a geodesic connecting $p$ to $h.$  As $h$ is in $X$, the $\PP$--Voronoi cell with nucleus $p,$ the ball $\HB(h,r)$ with $r=\dH(h,p)$ is disjoint from $\PP.$  Hence any $A \in \Gamma$ intersecting $\HB(h,r-D)$ is empty.  If $r \leq D,$ then $\dH(L_h,L_p) \leq D$ and so $L_h$ and $L_p$ are connected in $\mathcal{A}$.  Hence, the only portion of $\gamma$ not necessarily covered by cells from $\mathcal{E}(S)$ is the initial segment of length $D.$  In particular, the distance from $L_p$ to $\HB(h,r-D)$ is at most $D,$ and hence $L_h$ and $L_p$ are connected in $\mathcal{A}$ through $\mathcal{E}(S).$
\end{proof}

Furthermore, if we enlarge these discretizations of some Voronoi cells, we do not cover many more cells of $\mathcal{I}(\PP).$  Let $\dA$ be the graph distance in $\mathcal{A}.$
\begin{lemma}
\label{lem:balls}
For any $M \in \N,$ there is a $C=C(M) < \infty$ so that for all $S \subset \PP$ the following holds.
Let $W = \left\{ A \in \Lattice : \dA(A, \mathcal{E}(S) \cup \mathcal{I}(S)) \leq M \right\}$
then
\[
  |\left\{ A \in W : |A \cap \PP| > 0 \right\}| \leq C | \mathcal{I}(S)|.
\]
\end{lemma}
\begin{proof}
  Let $p \in S$ be arbitrary, 
  and let $X$ be the $\mathcal{S}$--Voronoi cell with nucleus $p.$
  Suppose $A \in \mathcal{E}(S)$ is an empty cell intersecting $X$ at a point $q.$
  As $q$ is contained in the $\mathcal{S}$--Voronoi cell with nucleus $p,$ 
  with $r = \dH(q,\mathcal{I}(p))$ we have that $\HB(q,r) \cap \mathcal{S} = \emptyset.$ 
  Furthermore, 
  \[
    \HB(q,r) \supset \bigl( \cup \left\{ B \in \Lattice : \dA(B,A) < r/(4D) \right\} \bigr).
  \]
  Hence, any $A \in \mathcal{E}(S)$ intersecting $X$ for which $\dH(A,\mathcal{I}(p)) > 4DM$ has the property that every $A \in \Lattice$ within $\dA$--distance $M$ contains no points of $\PP.$ Thus, there is a constant $C=C(M)$ so that
  \[
    |\left\{ A \in W : A \cap X \neq \emptyset, |A \cap \PP| > 0 \right\}| \leq C,
  \] 
  and summing over all $\mathcal{I}(S)$ proves the Lemma.
\end{proof}


These tools combine to give a proof Theorem \ref{thm:bmo_lattice} for showing when the Delaunay graph of $\PP$ has anchored nonamenability.  Before launching into the proof, we remark that \eqref{eq:bmd} or \eqref{eq:bmc}, we may replace $\mathcal{L}_A$ by the connected subsets of $\mathcal{A}$ which contain $A,$ as for any connected set $J$ in $\mathcal{A},$ we may find a connected set $J' \supset J$ in $\mathcal{L}_A$ for which $|J'| < C J$ for some constant $C$ depending only on $\Lattice.$ 

\begin{proof}[Proof of Theorem \ref{thm:bmo_lattice}]
  We wish to show there is a $\delta > 0$ so that if $S \ni x$ is any finite set in $\PP$ which is connected in the Delaunay graph on $\PP$ to an anchor point $\rho \in \PP$ then 
  \[
    |\del_{\text{out}} S| > \delta |S|.
  \]
  Let $A$ be an anchor cell of $\Gamma$ that contains $\rho.$ Let $J$ be $\mathcal{E}(S) \cup \mathcal{I}(S),$ which contains $A$.  By Lemma~\ref{lem:connected}, $J$ is connected in $\mathcal{A}$.  By the hypothesis that $\PP$ has anchored bounded density (recall \eqref{eq:bmd}), we have that there is a $\delta_0 > 0$ so that
  \begin{equation}\label{eq:thm1}
    |J| > \delta_0 |S|. 
  \end{equation}
  By the hypothesis that $\PP$ has anchored bounded vacancy, we have there is an $M$ and a $\delta_1 > 0$ so that
  \begin{equation}\label{eq:thm2}
    |\{B \in \Lattice : \dA(B,J) \leq M, |B \cap \PP| > 0\}|
    > \delta_1 |J|.
  \end{equation}
  By Lemma~\ref{lem:balls}, there is a $\delta_2 >0$ so that
  \begin{equation}\label{eq:thm3}
    |\mathcal{I}(S)|
    >
    \delta_2
    |\{B \in \Lattice : \dA(B,J) \leq M, |B \cap \PP| > 0\}|.
  \end{equation}
  By Lemma~\ref{lem:vboundary}, there is a $\delta_3 >0$ so that 
  \[
    |\del_{\text{out}} S| > \delta_3 |\mathcal{I}(S)|.
  \]
  Hence combining this with \eqref{eq:thm3}, \eqref{eq:thm2} and \eqref{eq:thm1}, we have the desired conclusion.
\end{proof}

We also show that under the additional control of Theorem \ref{thm:bmo_lattice2}, the degrees are under control.
\begin{proposition}\label{prop:degrees}
  Suppose that $\PP \subset \mathbb{H}^d$ has both anchored $p$-exponential density and anchored $q$-exponential vacancy for $\tfrac{1}{p} + \tfrac{1}{q} < \tfrac{1}{2}$, then the Delaunay graph $G$ on $\PP$ has 
  \[
    \sup_{\substack{S \ni \rho \\ S \text{ connected }}} \frac{\Vol[S](S)}{|S|} < \infty,
  \]
  with $\Vol[S](S)$ twice the number of edges from $S$ to $S.$
\end{proposition}
\begin{proof}

  Let $E$ be the empty cells intersecting $\mathcal{I}(S),$ and let $J$ be $E \cup \mathcal{I}(S).$ By Lemma~\ref{lem:connected}, $J$ is connected to $A$ in $\mathcal{A}.$  Using anchored bounded density and anchored bounded vacancy, as in the proof of Theorem \ref{thm:bmo_lattice}, we have $\delta |J| < |\mathcal{I}(S)|$ for some $\delta.$  

  Recall that for an edge to exist between $s,t \in \mathcal{S},$ there must be an open hyperbolic ball $U$ with $s,t \in \partial U$ with $U \cap \mathcal{S} = \emptyset,$ and we will say that $U$ \emph{certifies} the edge from $s$ to $t.$  Note the center of this ball is contained in a cell of $J,$ as it is on the boundary of a Voronoi cell with nucleus in $S.$
    Hence if we define
    \(
    h(x,J)
    =h(x,J,\mathcal{S})
    \)
    as the number of $\mathcal{S}$--Delaunay edges within $J,$ with one vertex $x,$ and with the additional property that the edge is certified by an open ball centered in a cell of $J$,
    then it suffices to show that
    \[
      \sup_{J \in \mathcal{L}_A} \frac{\sum_{B \in J} \sum_{x \in B}h(x,J)}{|J|} < \infty.
    \]

    For any cell $B \in J,$ suppose there is a $U$ certifying an edge of $\mathcal{S}$ and having a center in $B.$  Suppose the radius of this ball is $r.$  Then it follows 
    \[
      r - D < \dH(\mathcal{S},B) < r.
    \]
  Hence, if we define $\Delta B$ to be the collection of $\Lattice$ cells $X$ which have 
  \[
    \dH(X,B) - 2D < \dH(\mathcal{S},B) < \dH(X,B),
  \]
  then every edge certified by a ball with center in $B$ has both endpoints in a cell in $\Delta B.$  Hence we can estimate
  \begin{equation}\label{eq:b1}
    \sum_{B \in J} \sum_{x \in B}h(x,J)
    \leq 
    2
    \sum_{B \in J} \biggl( \sum_{X \in \Delta B \cap J} |X \cap \mathcal{S}| \biggr)^2.
  \end{equation}
  Let $\lambda > e^{2(d-1)}$ and apply H\"older to the previous bound, from which we get
  \[
    \biggl(\sum_{B \in J} \lambda^{p\dH(\mathcal{S},B)}\biggr)^{1/p} \biggl(\sum_{B \in J} \biggl(\lambda^{-\dH(\mathcal{S},B)/2} \sum_{X \in \Delta B \cap J} |X \cap \mathcal{S}| \biggr)^{2q}\biggr)^{1/q}.
  \]
  First term is controlled from exponential mean vacancy.  For the second, we apply convexity to get the bound  \[
    \sum_{B \in J}
    \biggl(\lambda^{-\dH(\mathcal{S},B)/2} \sum_{X \in \Delta B \cap J} |X \cap \mathcal{S}| \biggr)^{2q}
     \leq
     \sum_{B \in J}
     |\Delta B|^{2q-1} \lambda^{-q\dH(\mathcal{S},B)}
     \sum_{X \in \Delta B \cap J} |X \cap \mathcal{S}|^{2q}
  \]
  Rearranging the sum, we have
  \[
     \sum_{B \in J}
     |\Delta B|^{2q-1} \lambda^{-q\dH(\mathcal{S},B)}
     \sum_{X \in \Delta B \cap J} |X \cap \mathcal{S}|^{2q}
     \lesssim
     \sum_{X \in J}|X \cap \mathcal{S}|^{2q}
     \sum_{k=1}^\infty e^{2q(d-1)k}\lambda^{-qk},
  \]
  which is bounded by $C|J|$ by assumption.
\end{proof}

\section{Application to Poisson point processes}
\label{sec:ppp}

We make some simple observations which allow Theorem~\ref{thm:bmo_lattice} to be useful, for example when applied to Poisson points.  We begin by observing that on account of $\mathcal{A}$ having bounded degree (in fact being regular), there is a constant $\Delta$ so that
\begin{equation}
  \label{eq:LAgrowth}
  \left|
  J \in \mathcal{L}_A : |J| = n
  \right| \leq \Delta^{n-1}~\forall n \in \N.
\end{equation}
\noindent We also observe that enlarging a set $J \in \mathcal{L}_A$ enlarges it by a multiple that can be made arbitrarily large depending on $M.$
\begin{lemma}
  \label{lem:boundary}
  For any $R>0,$ there is an $M$ so that
  \[
    \inf_{ J \in \mathcal{L}_A} \frac{|\{B \in \Lattice : \dA(B,J) \leq M\}|}{|J|} > R,
  \]
\end{lemma}
\begin{proof}
  This follows as $\mathcal{L}_A$ is an expander, i.e.\ 
  \[
    \inf_{\substack{V \subset \mathcal{L}_A \\ |V| <\infty}} \frac{|\partial V|}{|V|} > 0,
  \]
  which can be deduced from the isoperimetric inequality for $\Htwo^d.$ 
\end{proof}

We now check that Corollary \ref{cor:ppp_lattice} follows from Theorem~\ref{thm:bmo_lattice}, which is to say that $\PPP$ has anchored bounded density and anchored bounded vacancy. 

\begin{lemma}
  \label{lem:ppp}
  A stationary Poisson point process has anchored bounded density and anchored bounded vacancy.
\end{lemma}

\begin{proof}
  We check each criterion of Theorem \ref{thm:bmo_lattice} separately.  Let $\lambda\cdot dV$ be the intensity of $\PPP,$ with $dV$ the volume measure on $\Htwo^{d}.$  By the properties of the Poisson process, for $J \in \mathcal{L}_A$ of cardinality $n$,
  \[
    |\PPP \cap (\cup_{B \in J} B)| \lawequals \Poisson(\lambda\cdot V \cdot n).
  \]
  In particular, there is a constant $C > 0$ so that for all $t \geq 0,$ $n \in \N$
  \[
    \Pr\left[ 
    |\PPP \cap (\cup_{B \in J} B)| 
    \geq (1+t)\lambda V n
    \right] \leq e^{-Ct \lambda V n}.
  \]
  Therefore, picking $t$ sufficiently large,
  \[
    \sum_{n=1}^\infty
    \Pr\left[ \exists J \in \mathcal{L}_A : |J| = n,
    |\PPP \cap (\cup_{B \in J} B)| 
    \geq (1+t)\lambda V n
    \right]
    \leq 
    \sum_{n=1}^\infty
    \Delta^n e^{-Ct \lambda V n} < \infty.
  \]
  So, by Borel--Cantelli,
  \[
    \sup_{ J \in \mathcal{L}_A} \frac{|\PPP \cap (\cup_{B \in J} B)|}{|J|} < \infty,
  \]
  almost surely.

  Conversely, for $J \in \mathcal{L}_A$ of cardinality $n,$
  \[
    |\{B \in \Lattice : \dA(B,J) \leq M, |B \cap \PPP| > 0\}|
    \lawequals
    \Binomial(X,1-\exp(-\lambda V)),
  \]
  where $X = |\{B \in \Lattice : \dA(B,J) \leq M\}.$  Hence, there is a constant $\delta >0$ so that
  \[
    \begin{aligned}
    &\Pr
    \left[ 
    |\{B \in \Lattice : \dA(B,J) \leq M, |B \cap \PPP| > 0\}|
    \leq X(1-\exp(-\lambda V))/2
    \right]  \\
    &\leq e^{-\delta X(1-\exp(-\lambda V))}.
  \end{aligned}
  \]
  By Lemma~\ref{lem:boundary}, by choosing $M$ sufficiently large, we can guarantee that
  \[
    \delta X(1-\exp(-\lambda V)) \geq  n (\log \Delta + 1),
  \]
  so that once more we can sum over all $J \in \mathcal{L}_A$ of cardinality $n$ to conclude that
  \[
    \Pr\left[ \exists J \in \mathcal{L}_A : |J| = n,
       |\{B \in \Lattice : \dA(B,J) \leq M, |B \cap \PPP| > 0\}|
     \right]
     \leq e^{-n}.
  \]
  Once more applying Borel--Cantelli, the desired conclusion holds. 
\end{proof}
\section{Berezin process}
\label{sec:hbr}
In this section we will build and partially characterize a family of stationary determinantal processes on $\Htwo^d.$  We do this, as to our knowledge only one stationary determinantal point process, the hyperbolic GAF has been developed (ignoring thinnings thereof).  We shall construct a one-parameter family of these processes which can be indexed by intensity.  For background on determinantal processes, see \cite{ShiraiTakahashi,HKPV,Soshnikov,AGZ}.

We remark from the outset that the standard method of constructing determinantal point processes would be to choose a reproducing kernel for some Hilbert space of functions on $\mathbb{H}^d.$  For example, the classic Bergman kernel $\frac{1}{(1-w\bar{z})^2}$ which defines the point process of zeros of the hyperbolic GAF is the reproducing kernel for the holomorphic functions of the disk which are in $L^2(|dz|^2)$.  A natural candidate to generalize this to higher dimensional real hyperbolic space might be to choose the reproducing kernel for some space of hyperbolic harmonic functions.\footnote{Nota bene, in conformal models of hyperbolic space in dimension $2,$ Euclidean and hyperbolic harmonic functions agree. In higher dimensional real hyperbolic space, this no longer holds.} 

Moreover, it would suffice to have any Hilbert space $\mathcal{H}$ of functions on $\mathbb{H}^d$ with a stationary inner product, i.e.\ one for which $\langle f , g\rangle = \langle f \circ \psi, g \circ \psi\rangle$ for any isometry $\psi,$ on which the evaluation map $f \mapsto f(x)$ is a bounded operator for all $x \in \mathbb{H}^d.$  The natural candidate of stationary harmonic functions on $L^2(dV)$ with $dV$ given by Haar measure unfortunately turns out to be the space $\{ 0 \}$ \cite[Corollary 10.4.2]{Stoll}.  Further, it is an open question to find such a Hilbert space of harmonic functions.  See \cite[Exercise 10.8.13]{Stoll}.
%
%
So, we will look at a determinantal process which is \emph{not} given by a reproducing kernel, but which seems to have some nice structural properties, including some natural interactions with hyperbolic harmonic functions (see Section \ref{sec:berezinthings}). 

We will work in the Poincar\'e ball model $B \subset \R^d,$ which we will identify with $\mathbb{H}^d.$
Let $\phi_a$ be the M\"obius automorphism of the ball exchanging $a$ with $0$.\footnote{See \cite{Ahlfors} for a thorough exposition on computations of higher dimensional M\"obius transformations. This is $T_a$ in \cite{Ahlfors}.}
When $d=2,$ with complex notation, 
\(
  \phi_a(z) = \frac{a-z}{1-\overline{a} z}.
\)
In higher dimensions, this can be given as \cite[(2.1.6)]{Stoll}
\begin{equation}\label{eq:mobius}
  \phi_a(x) \coloneqq \frac{a|x-a|^2 + (1-|a|^2)(a-x)}{[x,a]^2},\quad [x,a]^2 = 1-2\langle x,a\rangle+|a|^2|x|^2. 
\end{equation}

The modulus of this M\"obius transformation can be deduced from the identity \cite[(2.1.7)]{Stoll}, which gives
\begin{equation}
  1-|\phi_a(x)|^2 
  =\frac{(1-|x|^2)(1-|a|^2)}{[x,a]^2}.
  \label{eq:mobiusmodulus}
\end{equation}
We will also use that the modulus $|\phi_a(x)|$ is invariant in that for any $b \in B$
\begin{equation}
  \label{eq:invariance}
  |\phi_a(x)|^2 
  =
  |\phi_{\phi_b(a)}(\phi_b(x))|^2
\end{equation}
see \cite[(38)]{Ahlfors}.  From this it also follows that $|\phi_a(x)| = |\phi_x(a)|.$
The hyperbolic distance $\dH(a,b)$ can be expressed as
\begin{equation}
  \label{eq:hyp}
  \dH(a,b) = \log \biggl( \frac{1+|\phi_a(b)|}{1-|\phi_a(b)|}\biggr).
\end{equation}
We also observe the identity $\cosh(\dH(x,y))^{-1} = 1- |\phi_x(y)|^2,$ where we use $\cosh, \sinh$ to denote the hyperbolic cosine and sine respectively.

We begin with an observation on positive definiteness.  To agree with convention, we will let $\sigma = \tfrac{d-1}{2},$ and we define the kernel
\begin{equation}\label{eq:kernels}
  \mathcal{K}_s(x,y) \coloneqq (1-|\phi_x(y)|^2)^{\sigma(1+s)}= \cosh(\dH(x,y))^{-\sigma(1+s)}.
\end{equation}
which is therefore invariant and symmetric.  Moreover:
\begin{lemma}\label{lem:pdefkernel}
  For any $s \geq 0$ the kernel $\mathcal{K}_s :\Htwo^d \times \Htwo^d \to \R_+$ is positive definite.
\end{lemma}
\begin{proof}
  The mapping $(x,y) \mapsto \log(\cosh(\dH(x,y)))$ has \emph{negative type} (see \cite[Proposition 7.3]{FarautHarzallah} or \cite[Section 2.6,Theorem 2.11.3]{BHV}; for the definition, c.f.\ \cite[Definition C.2.1]{BHV} where it is called \emph{conditionally negative type}).  As a consequence of Schoenberg's Theorem \cite[Theorem C.3.2]{BHV}, for any $t \geq 0,$
  \[
    e^{-t\log(\cosh(\dH(x,y)))}
  \]
  is positive definite.
  There is only to observe that
  \[
    \cosh(\dH(x,y)) = 1 + \frac{|x-y|^2}{(1-|x|^2)(1-|y|^2)} = \frac{[x,y]^2}{(1-|x|^2)(1-|y|^2)}
    =(1-|\phi_x(y)|^2)^{-1},
  \]
  using \eqref{eq:mobiusmodulus}.
\end{proof}

The kernel $\mathcal{K}_s$ is locally trace class on $L^2(dV)$ with $dV$ the hyperbolic volume measure
\begin{equation}
  dV(x) = \frac{dx}{c_d(1-|x|^2)^{d}},
\end{equation}
with $c_d$ the Euclidean volume of $B.$  Indeed the diagonal of $\mathcal{K}_s$ is $1$.

To define a determinantal point process, we must additionally be able to control the norm of the operator $\mathcal{K}_s.$ 
For $s > 1,$ this is a triviality, as from \eqref{eq:kernels}, $\mathcal{K}_s(x,y) \leq 2^{\sigma(1+s)}e^{-\dH(x,y)\sigma(1+s)},$ which decays faster then the volume growth exponent.  Hence $\int_{\Htwo^d} \mathcal{K}_s(x,y) \,dV(y)$ is uniformly bounded in $x.$  This implies the operator norm bound by the Schur test.  We will set $\mathcal{N}_{s,d}$ to be $\| \mathcal{K}_s\|_{op}.$
\begin{lemma}
  $\mathcal{N}_{s,d} < \infty$ if and only if $s > 0.$  Moreover, in the extremes, we have
  \[
    s^{-1} \lesssim
    \mathcal{N}_{s,d} \lesssim s^{-2}
    \text{ as } s \to 0
    \quad
    \text{and}
    \quad
    \mathcal{N}_{s,d} \asymp
    s^{-d/2}
    \quad
    \text{ as } s\to \infty.
  \]
\end{lemma}
\begin{proof}
  We again apply the Schur test.  Let $\beta \in (\sigma(1-s),\sigma)$ be fixed and let $p(y) = \cosh( \dH(0,y))^{-\beta}$ for all $y \in \Htwo^d.$
We must show that
\[
  \sup_{x \in \Htwo^d} \frac{1}{p(x)}\int_{\Htwo^d} \mathcal{K}_s(x,y)p(y) \,dV(y) < \infty.
\]
From the hyperbolic law of cosines, if we set $\alpha$ as the angle made between $(0,x,y),$ then
\[
  \cosh( \dH(0,y) ) = \cosh( \dH(0,x))\cosh( \dH(x,y)) - \sinh( \dH(0,x))\sinh( \dH(x,y))\cos(\alpha).
\]
Then we can estimate
\[
  \cosh( \dH(0,y) ) \geq \cosh( \dH(0,x))\cosh( \dH(x,y))(1 - \cos(\alpha)_+)
\]
Hence if we integrate this in polar coordinates centered at $x,$
\begin{equation}\label{eq:norm1}
  \begin{aligned}
   &\int_{\Htwo^d} \mathcal{K}_s(x,y)p(y)\,dV(y) \\
   &\leq
   \cosh(\dH(x,0))^{-\beta}
   \int\limits_0^\infty
   \int\limits_{S^{d-1}}
   \cosh(r)^{-\sigma(1+s)-\beta}
   (1 - \cos(\alpha)_+)^{-\beta}
   f'(r)\,drd\vartheta.
 \end{aligned}
 \end{equation}
with $f'(r)=\sinh(r)^{d-1}.$
Under uniform measure on the sphere, the variable $\cos(\alpha)$ has the density
\[
  \frac{\Gamma(\sigma + \tfrac12)}{\sqrt{\pi}\Gamma(\sigma)} (1-x^2)^{{\sigma} - 1}\one[ |x| \leq 1] \,dx.
\]
Using the bound
\(
1-(x)_+ \geq \frac{1-x^2}{2},
\)
we conclude that for $\beta < \sigma.$
\[
  \begin{aligned}
   \int\limits_{S^{d-1}}
   (1 - \cos(\alpha)_+)^{-\beta}
   d\vartheta(\alpha)
   &\leq
   2^{\beta}
   \frac{\Gamma(\sigma + \tfrac12)}{\sqrt{\pi}\Gamma(\sigma)} 
   \int_{-1}^1
   (1-x^2)^{{\sigma} - 1 - \beta}\one[ |x| \leq 1] \,dx, \\
   &=2^{\beta}
   \frac{\Gamma(\sigma + \tfrac12)\Gamma(\sigma-\beta)}{\sqrt{\pi}\Gamma(\sigma)\Gamma(\sigma-\beta+\tfrac 12)}. 
 \end{aligned}
\]
We also have for $\sigma(1+s)+\beta > (d-1),$
\begin{equation}\label{eq:betaintegral}
  \int_0^\infty \sinh(r)^{d-1}\cosh(r)^{-\sigma(1+s)-\beta}\,dr
  =
  \frac{\Gamma(\sigma+\tfrac 12)\Gamma(\tfrac{\sigma(s-1)}{2}+\tfrac{\beta}{2})}{2\Gamma(\tfrac{\sigma(1+s)}{2}+\tfrac{\beta}{2}+\tfrac12)}.
\end{equation}
Applying these bounds to \eqref{eq:norm1}, we conclude that for any $\beta \in (\sigma(1-s),\sigma)$ 
\begin{equation}\label{eq:normupper}
  \int_{\Htwo^d} \mathcal{K}_s(x,y)\tfrac{p(y)}{p(x)}\,dV(y)
  \leq
  2^{\beta}
  \frac{\Gamma(\sigma + \tfrac12)\Gamma(\sigma-\beta)}{\sqrt{\pi}\Gamma(\sigma)\Gamma(\sigma-\beta+\tfrac 12)}
  \frac{\Gamma(\sigma+\tfrac 12)\Gamma(\tfrac{\sigma(s-1)}{2}+\tfrac{\beta}{2})}{2\Gamma(\tfrac{\sigma(1+s)}{2}+\tfrac{\beta}{2}+\tfrac12)}.
\end{equation}

We evaluate the asymptotics of this bound.
As $s\to 0,$ it is asymptotically optimal to pick $\beta = \sigma(1-\tfrac{s}{2}).$  For this choice we get that the bound shown asymptotically is
  \[
  2^{\beta}
  \frac{\Gamma(\sigma + \tfrac12)\Gamma(\sigma-\beta)}{\sqrt{\pi}\Gamma(\sigma)\Gamma(\sigma-\beta+\tfrac 12)}
  \frac{\Gamma(\sigma+\tfrac 12)\Gamma(\tfrac{\sigma(s-1)}{2}+\tfrac{\beta}{2})}{2\Gamma(\tfrac{\sigma(1+s)}{2}+\tfrac{\beta}{2}+\tfrac12)}
  \sim \frac{2^{\sigma}\Gamma(\sigma+\tfrac 12)}{\sqrt{\pi}\Gamma(\tfrac 12)\Gamma(\sigma)}\frac{4}{\sigma^2 s^2} \text{ as } s \to 0.
  \]
  Conversely, as $s \to \infty,$ we may take $\beta = 0,$ for which choice the upper bound becomes
  \[
    \mathcal{N}_{s,d}
    \leq
    \frac{\Gamma(\sigma+\tfrac 12)\Gamma(\tfrac{\sigma(s-1)}{2})}{2\sqrt{\pi}\Gamma(\tfrac{\sigma(1+s)}{2}+\tfrac12)}
    \sim \frac{\Gamma(\sigma+\tfrac 12)}{2\sqrt{\pi}} \biggl(\frac{\sigma s}{2}\biggr)^{-\sigma-\tfrac 12} \text{ as } s \to \infty.
  \]

We can as well consider the lower bound.  Here we just use $p(y)$ as a test function, for $\beta > \sigma.$  Then
\[
  \mathcal{N}_{s,d} \cdot \int p(x)^2\,dV(x) \geq \int \mathcal{K}_s(x,y) p(x)p(y)\,dV(x)\,dV(y).
\]
Using $\cosh(\dH(x,y)) \leq \cosh(\dH(0,x))\cosh(\dH(0,y))$ whenever $\cos(\alpha) > 0$ (which has $\vartheta$--probability $\tfrac 12$) and \eqref{eq:betaintegral}
\begin{equation} \label{eq:normlower}
  \mathcal{N}_{s,d}
  \geq 2^{-1}
  \biggl(\frac{\Gamma(\sigma+\tfrac 12)\Gamma(\tfrac{\sigma(s-1)}{2}+\tfrac{\beta}{2})}{2\Gamma(\tfrac{\sigma(1+s)}{2}+\tfrac{\beta}{2}+\tfrac12)}\biggr)^2
  \biggl(\frac{\Gamma(\sigma+\tfrac 12)\Gamma(-\sigma+\beta)}{2\Gamma({\beta}+\tfrac12)}\biggr)^{-1}.
\end{equation}
If $s < 0$ then taking $\beta=\sigma(-s+1)$, we have $\mathcal{N}_{s,d} = \infty.$  Likewise for $s=0,$ taking $\beta \downarrow \sigma$ causes the bound to diverge. 
For $s > 0,$ we take $\beta = \sigma(1+s),$ 
\[
  \mathcal{N}_{s,d}
  \geq 2^{-1}
  \frac{\Gamma(\sigma+\tfrac 12)\Gamma(\sigma s)}{2\Gamma({\sigma(1+s)}+\tfrac12)}.
\]
\end{proof}
\noindent For the specific choice of $s=\frac{1}{d-1}$, we will compute the norm in 
Proposition \ref{prop:norm}.
We also remark that for the Berezin kernels in the unit ball of $\C^d$, similar norms have been evaluated analytically.  See \cite{LiuZhou} and \cite{Dostanic}.  

\begin{remark}
  From the fact that $\mathcal{K}_{s+t} = \mathcal{K}_s\mathcal{K}_t \leq \mathcal{K}_s,$ for any $s,t >0,$ and the positivity of the kernel, it follows that $\mathcal{N}_s$ must in fact be decreasing in $s.$  
\end{remark}

Hence for $s > 1,$ the operator $\mathcal{K}_s/\mathcal{N}_{s,d}$ is a locally trace class, non--negative contraction on $L^2(dV).$  From the Machi--Soshnikov theorem, \cite{Soshnikov}, it therefore defines a determinantal point process $\Delta_{s,d}$ on $B.$  Moreover, the resulting point process is stationary on $\Htwo^d,$ as for any M\"obius transformation $\psi$ the law of $\psi^{-1}(\Delta_{s,d})$ is determinantal with kernel $\mathcal{K}_s(\psi(\cdot),\psi(\cdot))/\mathcal{N}_{s,d}$ acting on $L^2(V \circ \psi),$ and both the kernel and reference measure are invariant under $\psi.$ 

\begin{definition}
  \label{def:Ks}
  Define the \emph{Berezin} point process $\Delta_{s,d}$ as the determinantal point process on $\Htwo^d$ with kernel $\mathcal{K}_s/\mathcal{N}_{s,d}.$
\end{definition}
The name is given as the kernel $\mathcal{K}_s$ defines the \emph{Berezin} kernel.  In $2$--dimensions, the properties of this kernel and its resulting map are relatively well developed in \cite{Hedenmalm}.  In higher dimensions, there are some properties developed in \cite[Chapter 10]{Stoll}.  There is also a spectral expansion developed for Berezin kernels in hyperbolic spaces \cite{vanDijkHille, Neretin1, Neretin2}.  We highlight a few nice properties in Section \ref{sec:berezinthings}. 

\begin{figure}
  \begin{subfigure}[t]{0.45\textwidth}
    \includegraphics[width=\textwidth]{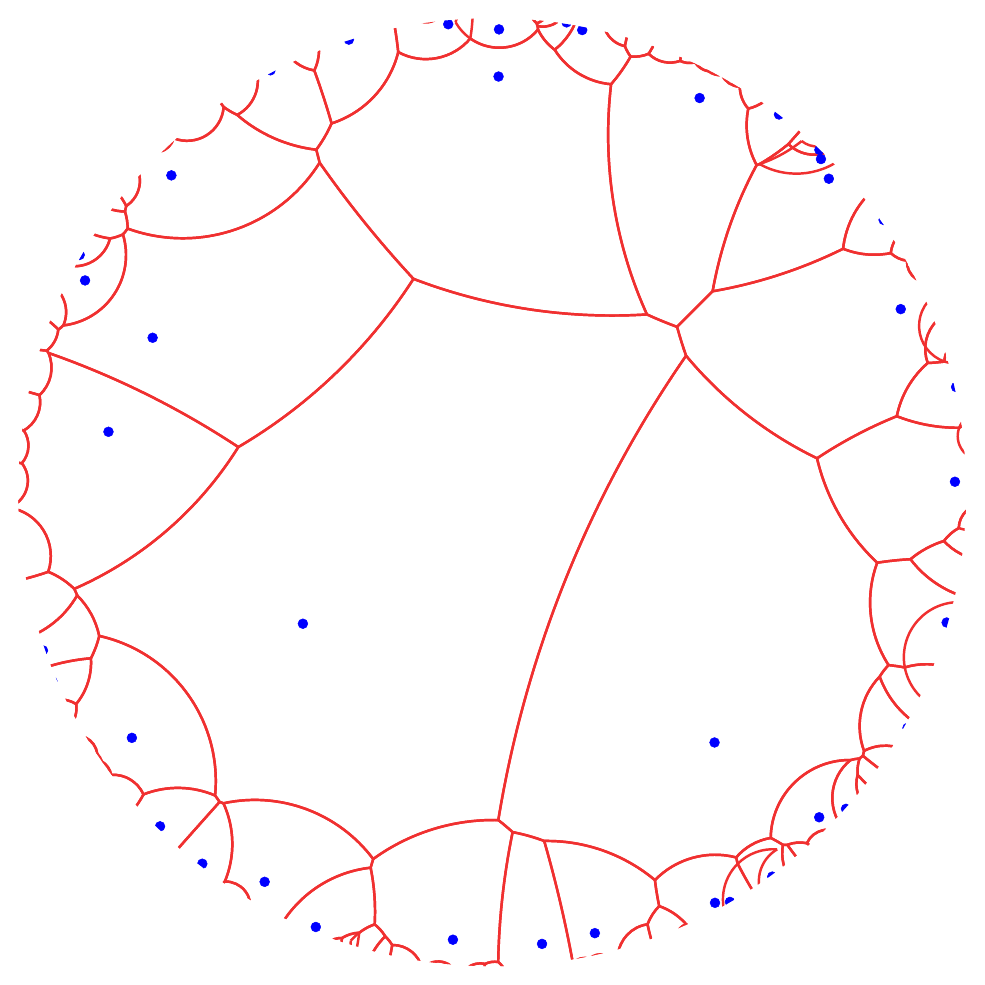}
    \caption{$s=3.$}
  \end{subfigure}
  \begin{subfigure}[t]{0.45\textwidth}
    \includegraphics[width=\textwidth]{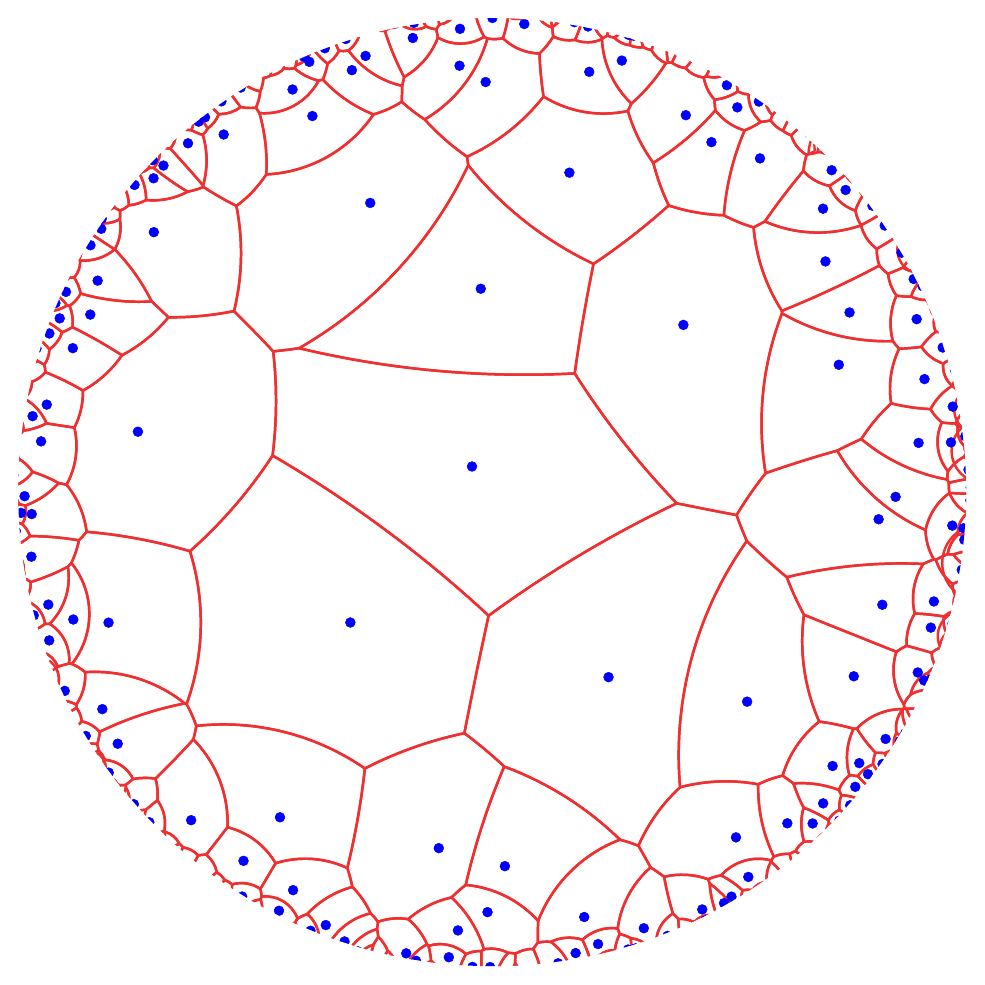}
    \caption{$s=10.$} 
  \end{subfigure}
  \begin{subfigure}[t]{0.45\textwidth}
    \includegraphics[width=\textwidth]{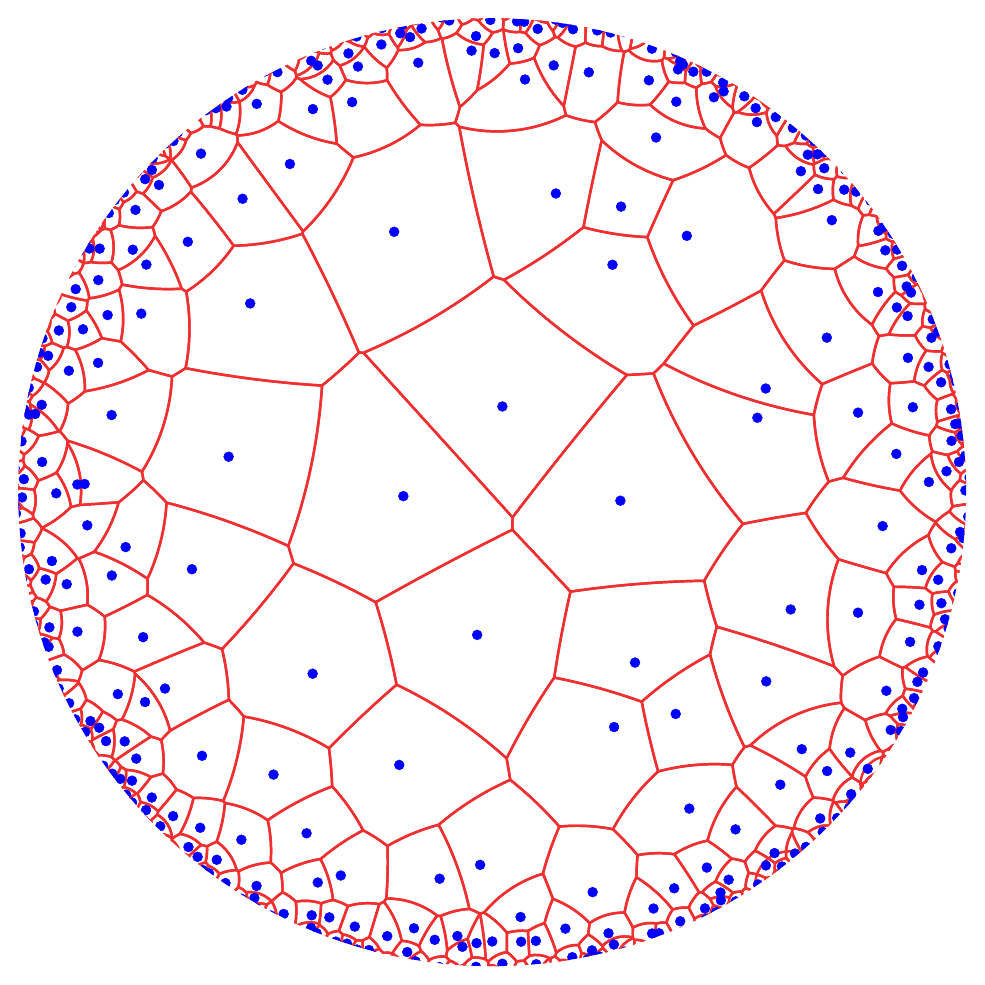}
    \caption{$s=20.$} 
  \end{subfigure}
  \begin{subfigure}[t]{0.45\textwidth}
    \includegraphics[width=\textwidth]{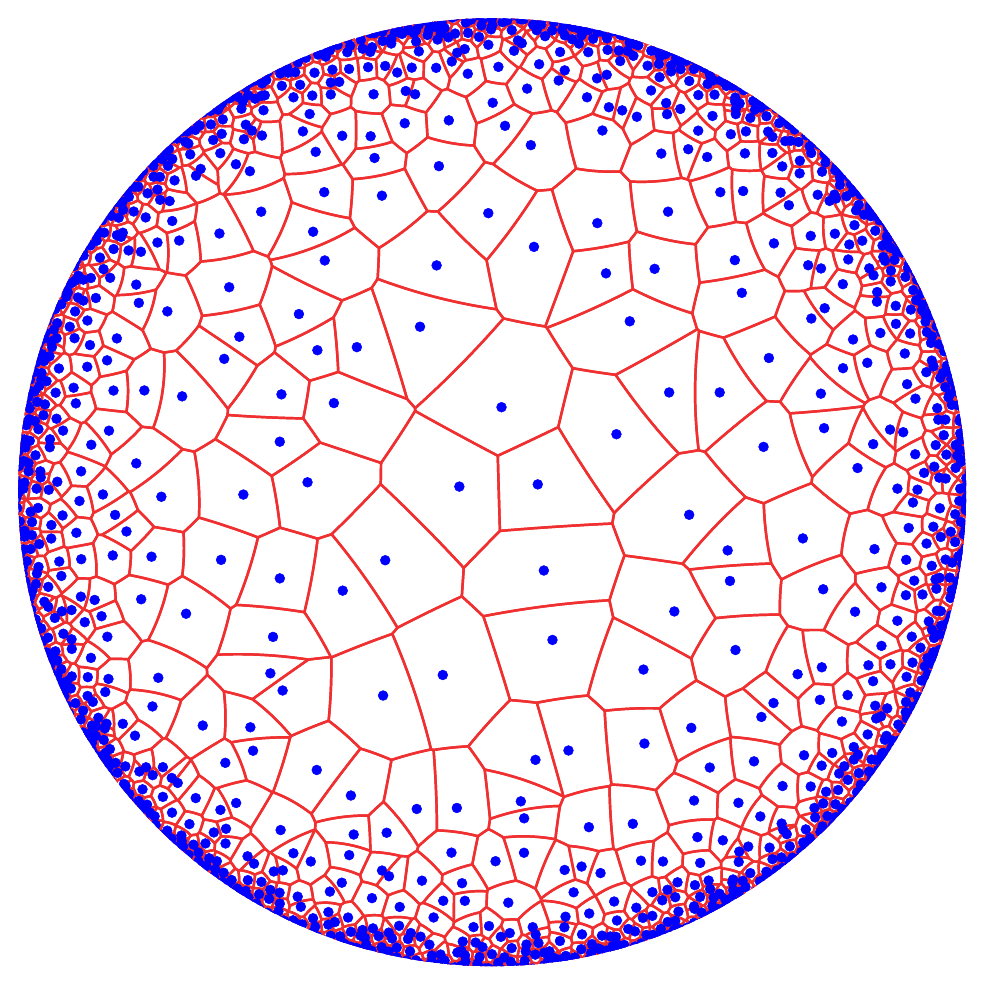}
    \caption{$s=80.$} 
  \end{subfigure}
  \caption{Four simulations of Berezin point process on Euclidean balls of size $0.99$ in the Poincar\'e model.  Norm taken as $\Gamma^2(\tfrac{s}{2})/\Gamma^2(\tfrac{s+1}{2})$ (c.f.\ Conjecture \ref{conj:norm}).  Subsampled from 40000 points hyperbolically uniform points, by defining a discrete Berezin kernel.  The discrete kernel can be identified with the action of the Berezin kernel on functions which are constant on the Voronoi cells formed by the 40000 seed points.}
  \label{fig:berezin}
\end{figure}

\subsection{ Relation to other kernels }

It may seem we have made an arbitrary choice in considering $\mathcal{K}_s(x,y) = \cosh(x,y)^{-\sigma(1+s)}.$  For example, the hyperbolic metric $\dH(x,y)$ is also negative type in any dimension, which follows as a corollary from the negativity of $\log\cosh(\dH(x,y))$ (see the discussion following \cite[Corollary 7.4]{FarautHarzallah}).

There is however a structure theorem defining all invariant positive definite kernels, which we will now describe, and which will show that in some sense there is only a $1$--parameter family of kernels possible.  
Define the invariant Laplacian $\Delta_h(f)(x) = \Delta (f \circ \phi_x) (0),$ which is up to a scalar multiple, the Laplace--Beltrami operator in coordinates (\cite[Remark 3.1.3]{Stoll}).
Define the kernel on $B \times S^{d-1}$ given by
\[
  P(a,\zeta)^{\sigma(1+s)} = \biggl(\frac{1-|a|^2}{[a,\zeta]^2}\biggr)^{\sigma(1+s)}.
\]
For $s=1,$ this defines the hyperbolic Poisson kernel on $B \times S^{d-1}$
\cite[(5.1.5)]{Stoll}.
More generally, from \cite{Helgason}, the mapping
\[
  f \mapsto \int_{S^{d-1}}P(\cdot ,\zeta)^{\sigma(1+s)}f(\zeta)\,d\vartheta(\zeta)
\]
for $f \in L^1(S^{d-1})$ maps densely into the space of all eigenfunctions of $\Delta_h$ with eigenvalue $4\sigma^2(s^2-1).$

From these eigenfunctions, we are particularly interested in the radial ones.
Define for complex $s$
\begin{equation}\label{eq:spherical}
  \Phi_{s}(r) = 
  (1-r^2)^{(1-s)\sigma} {}_2F_1((1-s)\sigma, -s\sigma + \tfrac 12 ; \sigma + \tfrac{1}{2} ; r^2),
\end{equation}
in terms of the Gauss hypergeometric function.
Then
\[
  \Phi_s(a) = 
  \int_{S^{d-1}}
  P(a,\zeta)^{\sigma(1+s)}
  \,d\vartheta(\zeta).
\]
We can also represent this function in an intrinsic way.  By using the hyperbolic law of cosines, if we let $\alpha$ be the angle formed between $\zeta \in S^{d-1}$ and $a$ from $0,$ then
\[
  \Phi_s(a) = 
  \tilde\Phi_s(\cosh(\dH(0,a))) 
  =\int_{S^{d-1}} ( \cosh(\dH(0,a)) -\sinh(\dH(0,a)) \cos(\alpha))^{-\sigma(1+s)}\,d\vartheta. 
\]
(c.f.\ \cite[$\omega_s$, p.206]{FarautHarzallah}, \cite[(4.42)]{Gangolli}).
We also need the function
\begin{equation}\label{eq:GQF}
  \tilde Q( \cosh(\dH(0,a)))
  =
  \int\limits_{S^{d-1}} \log( \cosh(\dH(0,a)) -\sinh(\dH(0,a)) \cos(\alpha))\,d\vartheta, 
\end{equation}
which is proportional to the derivative of $\Phi_s(a)$ in $s$ at $s=-1.$

As an analogue of the L\'evy-Khintchine formula (also called the L\'evy-Khintchine-Schoenberg formula), we have the following representation due to \cite{Gangolli} (c.f.\ \cite{FarautHarzallah}).
\begin{theorem}[Gangolli-Faraut-Harzallah representation]
  \label{thm:gangulli}
  Any invariant symmetric positive definite kernel $K(x,y)$ has the form $e^{-\tilde{\psi}(\cosh(\dH(x,y)))}$ where
  \[
    \tilde \psi = c \tilde Q 
    + \int_{-1}^0 (1 - \tilde\Phi_s)\,d\mu_1(s) 
  + \int_{i\R} (1 - \tilde\Phi_s)\,d\mu_2(s)
  \]
  for $c \geq 0,$ $\mu_1$ and $\mu_2$ positive measures such that $\int (1+s)d\mu_1 < \infty$ and such that $\mu_2$ is symmetric and finite.
\end{theorem}

Note that in this representation, all $\tilde \Phi_s$ with $s \in (-1,0)$ and $s \in i\R$ are uniformly bounded.  Hence for any $\tilde \psi$ having the Gangolli representation for $c>0,$
\[
  \tilde \psi(\cosh(0,x)) \sim c\tilde Q(\cosh(0,x)) \sim c\log(\cosh(\dH(0,x)))
  \quad
  \text{ as }
  |x|\to \infty.
\]
Thus we see that any invariant positive definite kernel is in some sense asymptotic to $\mathcal{K}_s$ for some choice of $s.$  Moreover, we note that for any pair $K_{\psi_1}$ and $K_{\psi_2}$ of invariant kernels, which are bounded on $L^2(dV),$ they commute.

\begin{corollary}
  Let $\tilde\psi$ be given by Theorem \ref{thm:gangulli} with $c > \sigma.$  Then $K_\psi(x,y) = e^{-\tilde\psi(\cosh(\dH(x,y)))}$ has finite operator norm $\mathcal{N}_\psi,$ and so $K_\psi/\mathcal{N}_\psi$ defines a stationary determinantal point process on $\mathbb{H}^d.$
\end{corollary}
\begin{proof}
  As $c > \sigma,$ there is $s>0$ so that $c=\sigma(1+s)$.
  Then the kernel $K_\psi(x,y)/K_{s}(x,y)$ is bounded from above and below uniformly, and so the estimate of Lemma \ref{lem:pdefkernel} applies to $K_\psi(x,y).$  
\end{proof}

We conclude by noting that in some cases, these functions can be explicitly computed.
In dimension $2,$ \cite[(4.45)]{Gangolli} we have the explicit representation
\begin{equation}\label{eq:d2Q}
  \tilde Q( \zeta)
  =\log\biggl(\frac{1+\cosh(\zeta)}{2}\biggr).
\end{equation}
In dimension $3,$ we have \cite[(4.57)]{Gangolli}
\begin{equation}\label{eq:d3Q}
  \tilde Q(\zeta)
  = \zeta\frac{\cosh(\zeta)}{\sinh(\zeta)}-1
  \quad
  \text{and}
  \quad
  \Phi_s(\zeta)
  =\frac{\sinh( \zeta(1+s))}{(1+s)\sinh(\zeta)}.
\end{equation}

\subsubsection*{Hermitian kernels}

The most studied determinantal processes in dimension $2$ are given by Hermitian kernels.  The zero set of the hyperbolic Gaussian analytic function \cite{PeresVirag, HKPV} has kernel $(1-w\bar{z})^{-2}$ acting on $L^2(B)$ under Lebesgue measure.  The zeros of $\det ( \sum_{k=0}^\infty G_k z^k)$ for iid $L \times L$ complex Gaussian matrices similarly produce a hermitian kernel $(1-w \bar z)^{-L-1}$ acting on $L^2(\,dx(1-x^2)^{(L-1)})$ \cite{Krishnapur} (c.f.\ \cite{JancoviciTellez,Comtet}, where these also appear).  See also \cite{DemniLazag} for a further generalization of this family having Hermitian kernels, and see the discussion in \cite{BufetovQiu} for a discussion of stationary determinantal point processes in higher dimensional real and complex hyperbolic spaces. 

There are some exact connections between these Hermitian kernels and the Berezin kernels.  The kernel
\[
  H_{L} \coloneqq \frac{(1-|w|^2)^{\tfrac{L+1}{2}}(1-|z|^2)^{\tfrac{L+1}{2}}}{(1-w\bar{z})^{L+1}}
\]
acting on $L^2(dV)$ induces the same determinantal process as those of \cite{Krishnapur}.  Then $|H_{L}(w,z)| = \mathcal{K}_{L}(x,y)$ on taking $w=x_1+ix_2$ and $z=y_1+iy_2.$  As a corollary, the determinantal processes induced by $H_L$ and $\mathcal{K}_L$ have the same $2$-point function (up to a scaling).

While these Hermitian kernels seem to induce incompatible processes with those coming from real symmetric kernels, we have not been able to rule it out.  So we pose the following question:
\begin{question}
  Is there real symmetric kernel $K : \Htwo^2 \times \Htwo^2$ acting on $L^2(dV)$ that induces the same law as 
  \(H_L
  \)
  acting on $L^2(dV)$? If no, is there any precise relationship between the determinantal process with kernel (up to scaling) $\mathcal{K}_L$ and that with $H_L?$
\end{question}

\subsection{The relation of the Berezin kernels to harmonic functions}
\label{sec:berezinthings}

We highlight a few connections of the Berezin kernel to the Laplacian.  We begin with noting that harmonic functions are in fact generalized eigenfunctions of these kernels.
\begin{proposition}\label{prop:stoll}
  For a hyperbolic harmonic function $f$ for which $\int_{\mathbb{H}^d} |f(x)|(1-|x|^2)^{\sigma(1+s)}\,dV(x) < \infty$ and
  for all $x \in \mathbb{H}^d,$
  \[
    f(x) = \frac{1}{c_s}\int \mathcal{K}_{s}(x,y)f(y)\,dV(y)
    \quad
    \text{where}
    \quad
    c_s = \frac{d}{2}\frac{\Gamma(\tfrac{d}{2})\Gamma(\sigma(1+s)-1)}{\Gamma( \tfrac{d}{2} + \sigma(1+s)-1)}.
  \]
\end{proposition}
\noindent This is \cite[Theorem 10.1.3]{Stoll}.

In $2$--dimensions, an even stronger statement is possible.
Define the entire function in $z \in \C$, for any $s > 1,$ by
\[
  G_{s}(z) = \frac{s-1}{2}\prod_{k=1}^\infty\biggl( 1-\frac{z}{4(\tfrac{s-3}{2}+k)(\tfrac{s-1}{2}+k)}\biggr).
\]
Then formally, we have 
\[
  G_{s}(\Delta_h) \mathcal{K}_s(x,y) = \delta_{h,y}(x),
\]
where $\delta_{h,y}(x)$ is the distribution satisfying $\int_{\mathbb{H}^d} \delta_{h,y}(x)f(x)\,dV(x) = f(y)$ for all bounded continuous functions $f$ on $\mathbb{H}^d$ for which $\int\mathcal{K}_s(0,y) |f(y)|\,dV(y) < \infty$  (c.f.\ \cite[Section 2]{Hedenmalm}, wherein $\alpha = \frac{s-3}{2}$).  

In the case of $s=3,$ this series has a closed form given by
\[
  G_{3}(-4w(1-w)) = \frac{\sin(\pi w)}{\pi w(1-w)},
\]
(see \cite[Proposition 2.9]{Hedenmalm} or \cite{Rudin}).
Following the same argument as in \cite[Proposition 2.9]{Hedenmalm}, we can also give a formula for general $s>1$ 
\[
  G_s(-4w(1-w)) = \frac{\Gamma^2(\frac{s+1}{2})}{\Gamma(\frac{s}{2}+(w-\frac{1}{2}))\Gamma(\frac{s}{2}-(w-\frac{1}{2}))}. 
\]
Using this function we see that any eigenfunction $f$ of the Laplacian with eigenvalue $\lambda$ for which the integral of $\int_{\mathbb{H}^2} \mathcal{K}_s(0,x)|f(x)|\,dV(x) < \infty,$ it can be seen $\int_{\mathbb{H}^2} \mathcal{K}_s(x,y)f(y)\,dV(y) = f(x)/G_s(\lambda)$ (see \cite[Proposition 2.5,2.7]{Hedenmalm}).

The eigenfunctions corresponding to real eigenvalues are not in $L^2(dV)$ (see \cite[Corollary 5.5.8]{Stoll}), but the radial eigenfunction with eigenvalue $-1$ just barely fails, by logarithmic factors.  
Hence, we expect that there are almost-eigenfunctions in $L^2(dV)$ which approximately behave like this eigenfunction.  This would lead to a lower bound on $\mathcal{N}_{s.d}$ of $1/G_s(-1),$ which corresponds to taking $w = \frac 12.$  Hence, we conjecture:
\begin{conjecture}\label{conj:norm}
  For any $s \geq 1,$
  \[
    \mathcal{N}_{s,2} = \frac
    {\Gamma^2(\tfrac{s}{2})}
    {\Gamma^2(\tfrac{s+1}{2})}.
  \]
\end{conjecture}
\noindent We verify one case of this conjecture in Proposition \ref{prop:norm}.  While we have less evidence for it, we are tempted to conjecture this holds for $s>0$ as well.

\subsection{Series expansion for the Berezin kernel}

\newcommand{\twof}[2]{\genfrac{}{}{0pt}{}{#1}{#2}}

We will give an eigenfunction expansion for this kernel.  
Let $Z_m(x,y)$ be the $m$-th degree zonal function (see \cite[Chapter 5]{Axler}), which is a sum of Euclidean harmonic degree-$m$ homogeneous polynomials,
\begin{equation} \label{eq:zonal}
  Z_m(x,y) = \sum_{i=1}^{h_m} \phi_i^m(x)\phi_i^m(y).
\end{equation}
The dimension of the space of $m$-th degree harmonic homogeneous polynomials we denote by $h_m,$ which has the expression for $m \geq 2$ as (\cite[(5.17)]{Axler}) 
\[
  h_m = \binom{d+m-1}{d-1} - \binom{d+m-3}{d-1} \sim  \frac{2 m^{d-2}}{(d-2)!}, \text{ as } m \to \infty.
\]
Moreover the sum $\sum_{m=0}^\infty Z_m(x,y)$ is a reproducing kernel for $L^2(S^{d-1},d\vartheta),$ where $\vartheta$ is normalized surface measure on $S^{d-1}.$
Then from \cite{Minemura} (see also \cite{Sami} for a discussion), the function $P_h^{\sigma(1+s)}(ru,\zeta)$ admits a series expansion uniformly convergent on compacts and given by
\begin{equation}\label{eq:Ph}
  \begin{aligned}
  &P_h^{\sigma(1+s)}(r\omega,\zeta)
  =
  \sum_{m=0}^\infty \Phi_m^s(r) Z_m(\omega,\zeta), \\
  &\Phi_m^s(r) = 
  \frac{\Gamma(\tfrac d2)\Gamma(m+(1+s)\sigma)}{\Gamma(m + \tfrac d2) \Gamma( (1+s)\sigma ) }
  r^{m}(1-r^2)^{(1-s)\sigma} {}_2F_1\biggl(\twof{m + (1-s)\sigma,-s\sigma + \tfrac 12}{ m + \tfrac{d}{2}}; r^2 \biggr).
\end{aligned}
\end{equation}
We note that when $s\sigma-\tfrac 12 \in \N_0,$ this hypergeometric function becomes a polynomial.
\begin{lemma}\label{lem:Kexpansion}
There is convergent series expansion on compact sets for any $s > 0$ given by
\[
    \mathcal{K}_s(x,y) 
    =
    \biggl(\frac{(1-|x|^2)(1-|y|^2)}{1-|x|^2|y|^2}\biggr)^{\sigma(1+s)}
    \sum_{m=0}^\infty
    \Phi_m^s(|x||y|) Z_m(\tfrac{x}{|x|},\tfrac{y}{|y|}).
\]
\end{lemma}
\begin{proof}
  Let us observe first that for any nonzero $x,y \in B,$ $[x,y] = [x\cdot |y|, \tfrac{y}{|y|}],$ and therefore
  \[
    \begin{aligned}
    \mathcal{K}_s(x,y) 
    &= \biggl(\frac{(1-|x|^2)(1-|y|^2)}{[x,y]^2}\biggr)^{\sigma(1+s)} \\
    &= \biggl(\frac{(1-|x|^2)(1-|y|^2)}{[x|y|,\tfrac{y}{|y|}]^2}\biggr)^{\sigma(1+s)} \\
    &= P_h^{\sigma(1+s)}( x|y|,\tfrac{y}{|y|}) 
    \biggl(\frac{(1-|x|^2)(1-|y|^2)}{1-|x|^2|y|^2}\biggr)^{\sigma(1+s)}.
    \end{aligned}
  \]
  Thus on applying \eqref{eq:Ph}, we arrive at the claimed formula.
\end{proof}

\subsection{Specialization}
In this section, we investigate in more detail a special case of the Berezin kernel.

\subsection*{The quasi-Euclidean-harmonic case ($s=\frac{1}{d-1}$)}

We develop the properties of this particular choice a little further, as it is structurally simple.
In this case,
  \(
    \Phi_m^s(r) = r^m,
  \)
  and so we have 
  \begin{equation}
    \mathcal{K}_s(x,y) = 
    \frac{((1-|x|^2)(1-|y|^2))^{\frac{d}{2}}}{1-|x|^2|y|^2}
     \sum_{m=0}^\infty |x|^m|y|^m Z_m(\tfrac{x}{|x|},\tfrac{y}{|y|}).
   \end{equation}\label{eq:kernel}
   We call this quasi-Euclidean-harmonic as the kernel $\sum_{m=0}^\infty |x|^m|y|^m Z_m(\tfrac{x}{|x|},\tfrac{y}{|y|})$ gives the \emph{extended (Euclidean) Poisson kernel}, see \cite[8.11]{Axler}.

   We will define the kernel $\mathcal{R}_\theta(r,s)=\frac{r^{\theta-1}s^{\theta-1}}{1-rs},$ so that we have
   \begin{equation}\label{eq:Rm}
    \mathcal{K}_s(x,y) = 
    \biggl((1-|x|^2)(1-|y|^2)\biggr)^{\frac{d}{2}}
     \sum_{m=0}^\infty \mathcal{R}_{1+m/2}(|x|^2,|y|^2)Z_m(\tfrac{x}{|x|},\tfrac{y}{|y|}).
   \end{equation}

  Suppose as an illutration we want to compute the Hilbert-Schmidt norm of this restricted to a ball $d=2$.
  Then we have the formula
  \[
    \begin{aligned}
    \int_{\rho B}
    \int_{\rho B}
    \mathcal{K}_s(x,y)^2 \,dV(x)\,dV(y)
    &=
    \int_0^\rho
    \int_0^\rho
    \sum_{m=0}^\infty \frac{h_m r^{2m}s^{2m}}{(1-r^2s^2)^2} (d-1)^2r^{d-1}s^{d-1} \,drds \\
    &=
    \frac{(d-1)^2}{4}
    \int_0^{\rho^2}
    \!\!
    \int_0^{\rho^2}
    \frac{(1+rs)\,drds}{(1-rs)^{d+1}}
    (rs)^{\frac{d}{2}-1}
    \,drds.
    \end{aligned}
  \]
  In the second line we have used \cite[(8.11)]{Axler}.
  When $d=2,$ this can be evaluated explicitly to give
  \[
    \int_{\rho B}
    \int_{\rho B}
    \mathcal{K}_s(x,y)^2 \,dV(x)\,dV(y)
    =
    \frac{\rho^2}{1-\rho^2}.
  \]

  \subsubsection*{An equality for the norm $\mathcal{K}_{\tfrac{1}{d-1}}$}

  The radial parts of \eqref{eq:Rm} have an interesting connection to the Hilbert matrix. The kernel $\mathcal{R}_\theta(r,s) = \frac{(rs)^{\theta-1}}{1-rs}$ acting on $L^2[0,1]$ takes the monomial $s^{k}$ to the series $r^{\theta-1}\sum_{\ell=0}^\infty \frac{r^\ell}{k+\ell+\theta}.$  

    If we let $\mathcal{H}_\theta$ be the Hilbert matrix $(\mathcal{H}_\theta)_{k,\ell} = \frac{1}{k+\ell+\theta}$ for $k,\ell \geq 0,$ then for polynomials $p(z) = \sum a_k z^k$ we have that
    \[
      \iint \frac{p(r)p(s)r^{\theta-1}s^{\theta-1}\,drds}{1-rs} = \sum_{k,\ell,j \geq 0} a_j \frac{1}{j+\ell+\theta} \frac{1}{k + \ell+\theta} a_k = \langle a , \mathcal{H}_\theta^2 a\rangle,
    \]
    if we let $a$ be the coefficient vector of $p$ in $\ell^2(\N_0).$
    Likewise \( \int r^{\theta-1} p^2(r)\,dr = \langle a, \mathcal{H}_\theta a\rangle\).
    Also observe that for $\theta \geq \tfrac{d}{2},$
    \[
      \langle a, \mathcal{H}_\theta a\rangle
      =
      \int r^{\theta-1} p^2(r)\,dr
      \leq
      \int r^{\frac{d}{2}-1}p^2(r)\,dr
      = \langle a, \mathcal{H}_{\frac{d}{2}} a\rangle.
    \]
    
    Hence for any polynomial $p$,
    \[
      \frac{\iint p(r)\mathcal{R}_\theta(r,s)p(s)\,drds}
      {\int r^{\frac{d}{2}-1} p^2(r)\,dr}
	= 
	\frac{\langle a , \mathcal{H}_\theta^2 a\rangle}{\langle a, \mathcal{H}_{\frac{d}{2}} a \rangle}
	\leq
	\frac{\langle a , \mathcal{H}_\theta^2 a\rangle}{\langle a, \mathcal{H}_\theta a \rangle}.
    \]
    For $\theta > \tfrac 12,$ the matrix $\mathcal{H}_\theta$ has operator norm $\pi$ (\cite{Kato}), and hence we have that $\mathcal{R}_{\theta}$ has operator norm at most $\pi$ as well.  Note that in the case $\theta=1,$ we additionally derive that the norm of $\mathcal{R}_1$ is $\pi.$  As a corollary, using \eqref{eq:Rm}, we have shown:

    \begin{proposition}\label{prop:norm}
      The norm of the operator $\mathcal{N}_{\frac{1}{d-1},d} = \|\mathcal{K}_{\frac{1}{d-1}}\| = \pi,$ and hence the intensity of the Berezin process $\Delta_{\frac{1}{d-1},d}$ is $\frac{1}{\pi}dV.$ 
    \end{proposition}

    We expect an exact norm computation is possible for the other Berezin kernels:
    \begin{question}
      Is there an exact expression for $\mathcal{N}_{s,d}$ for other choices of $s$ and $d?$
    \end{question}

    \begin{remark}

  We can also diagonalize the kernel $\mathcal{K}_1$ (see \cite{Kalvoda} for details of what follows).  For $x \in \R,$ the function $F_{x,\theta}(t) \coloneqq (1-t)^{-\theta+\tfrac 12 +ix} { }_{2}F_1( \tfrac{1}{2}+ix, \tfrac{1}{2}+ix ; 1 ; t)$ (which is real valued) satisfies 
    \[
      \lim_{\rho \to 1}
      \int_0^\rho \frac{(rs)^{\theta-1}}{1-rs}
      F_{x,\theta}(s) \,ds = \frac{\pi}{\cosh(\pi x)} r^{\theta - 1}F_{x,\theta}(r), 
    \]
    at least for $\theta =1.$  Formally this holds for any $\theta >1,$ but the existence of the limit would seem to need an argument.
    Note that this makes $s^{\theta-1}F_{x,2\theta-1}(s)$ a generalized eigenfunction of $\mathcal{R}_{\theta}$.

    Except for $\theta<\frac{3}{2},$ these functions are not in $L^1,$ and for no $\theta$ no case can they be in $L^2.$  We also observe that from \eqref{eq:Ph} $F_{x,\theta}(r)$ is a multiple of the $d=2$ case of the spherical function $\Phi^{-2ix}_0(r).$

    Similar diagonalizations are possible for other Berezin kernels, (see \cite{vanDijkHille,Neretin1,Neretin2}).
  \end{remark}

  This leaves some hope that it is possible to diagonalize these operators when restricted to $[0,r]$ with $r \in (0,1)$ which would both make it simple to simulate these processes as well as give exact solutions to the vacancy probability (see Question \ref{q:vacancy}).
  \begin{question}
    Is there an exact expression for the diagonalization of the kernels $\mathcal{R}_m$ restricted to $[0,r]$ with $r \in (0,1)$ in terms of higher transcendental functions?  Is the analogous radial kernels for different $s?$
  \end{question}

    \subsection{Tail properties}

    In this section we let $\Pi$ be any stationary determinantal point process on $\mathbb{H}^d$ and suppose that $\Pi$ is generated by an invariant Hermitian positive definite kernel $K$ which is locally trace class.  By necessity, the intensity of $\Pi$ must be a multiple of $dV,$ which we denote by $\lambda dV$ for some $\lambda >0.$  Then the following hold:
    \begin{proposition}\label{prop:generalities}
      For a determinantal point process $\Pi$ on $\mathbb{H}^d$ with intensity $\lambda dV,$
      \begin{enumerate}
	\item For any compact $W \subset B,$
	  \[
	    \Pr( \Pi \cap W = \emptyset ) \leq e^{-\lambda \VolH(W)}.
	  \]
	\item For any compact set $W,$ if we let $\sigma^2 \geq \lambda \VolH(W) - \|\one_{W}K \one_{W}\|_{HS}^2$ then for all $t > 0,$
	  \[
	    \Pr( |\Pi \cap W| > t + \lambda \VolH(W) ) \leq e^{-\sigma^2 h(t/\sigma)},
	    \quad
	    h(u) = (1+u)\log(1+u) - u.
	  \]
	  We may always take $\sigma^2 = \lambda \VolH(W).$
	\item For any disjoint compact sets $W_1, W_2, \dots, W_n,$ the random variables $\left( |\Pi \cap W_j| \right)$ are negatively associated. 
      \end{enumerate}
    \end{proposition}
    \begin{proof}
      The first and second claim are consequences of the identity
      \[
	|\Pi \cap W| \lawequals \sum_{i=1}^\infty \Bernoulli(\lambda_i),
      \]
      where $1 \geq \lambda_1 \geq \lambda_2 \geq \cdots \geq 0$ are the eigenvalues of $K$ restricted to $W.$  As $\sum_{i=1}^\infty \lambda_i = \lambda \VolH(W)$ and $\sum_{i=1}^\infty \lambda_i^2 =  \|\one_{W}K \one_{W}\|_{HS}^2$ the claims follow from standard manipulations.  

      The final claim is a general property of determinantal point processes.  See \cite{Lyons}.
    \end{proof}
    These properties together allow the proof that the Poisson process has positive anchored expansion to go through for stationary determinantal point processes with minor modifications.

    We expect that the Berezin process has sharper tail behavior in both directions.  This can be inferred from the series expansion Lemma \ref{lem:Kexpansion}.  On the one hand, on a large ball, we expect there are many eigenvalues of these processes which are close to $1$ (in comparison to, e.g., the $1$--dimensional complex Bergman kernel):
    \begin{question}
      \label{q:vacancy}
      Is it the case that for any $s > 0,$
      \[
	\lim_{r \to \infty} \frac{\log \Pr( \Delta{s,d} \cap \HB(0,r)=\emptyset)}{\VolH(\HB(0,r))} = -\infty.
      \]
    \end{question}
    We also see that from the real analyticity of $\mathcal{K}_s$ on compacts, its eigenvalues $1 \geq \lambda_1 \geq \lambda_2 \geq \dots$ should have come stretched exponential decay.  Hence:
    \begin{question}
      For any fixed compact set $W \subset \Htwo^d,$ is there an $\alpha > 0$ so that
      \[
	\lim_{t \to \infty} \log(\Pr( |\Delta_{s,d} \cap W| > t))/t^{\alpha+1} = -\infty.
      \]
    \end{question}
    In both of the above properties, the parameter $s$ will play some role, in that it certainly parameterizes the intensity of the process.  It would be interesting to determine if that dependence is continuous in $s$ in that the rate functions only depend on $s$ through a multiplicative function of $s,$ or if the rates change more drastically as $s$ varies.

\printbibliography[heading=bibliography]

\end{document}